\documentclass{amsart}
\usepackage{graphicx}        
\usepackage[bottom]{footmisc}
\usepackage{amsmath,amsfonts}%
\usepackage{amsthm}%
\usepackage{mathrsfs}%
\usepackage{xcolor}%
\usepackage{comment}
\usepackage{subcaption}
\usepackage{enumerate}
\usepackage{tikz}
\usepackage[margin=1in]{geometry}

\newtheorem{definition}{Definition}
\newtheorem{theorem}{Theorem}
\newtheorem{remark}{Remark}
\newtheorem{lemma}{Lemma}
\newtheorem{corollary}{Corollary}

\def\p {\mathbf{p}}

\newcommand{\SG}{}
\def\SG/{split-and-glue}

\title{Nucleation-free independent graphs with implied nonedges}
\author{Jialong Cheng$^1$}
\author{Meera Sitharam$^1$}
\thanks{$^1$CISE Department, University of Florida, Gainesville, FL, USA}
\author{Ileana Streinu$^2$}
\thanks{$^2$Computer Science Department, Smith College, Northampton, MA, USA}
\author{William Sims$^{3,*}$}
\thanks{$^3$School of Information Technology, Illinois State University, Normal, IL, USA}
\thanks{$^*$Corresponding author, wsims3@ilstu.edu}
\date{}

\begin{document}

\begin{abstract}
We give inductive constructions of independent graphs that contain implied nonedges but do not contain any non-trivial rigid subgraphs, or \emph{nucleations}: some  of  the constructions and proofs apply to 3-dimensional abstract rigidity matroids with their respective definitions of nucleations and implied nonedges.  
The first motivation for the inductive constructions of this paper, which generate an especially intractable class of flexible circuits, is to illuminate further obstacles to settling Graver's maximality conjecture that the 3-dimensional generic rigidity matroid is isomorphic to   Whiteley's cofactor matroid (the unique maximal matroid in which all graphs isomorphic to $K_5$ are circuits).
While none of the explicit examples we provide refutes the maximality conjecture (since their properties hold in both matroids) the construction schemes are useful regardless whether the conjecture is true or false, e.g. for constructing larger (counter)examples from smaller ones.   
The second motivation is to make progress towards a polynomial-time algorithm for deciding independence in the abovementioned maximal matroid.  
Nucleation-free graphs with implied nonedges, such as the families constructed in this paper, are the key obstacles that must be dealt with for improving the current state of the art.
\end{abstract}

\maketitle



\section{Introduction}
\label{sec:introduction}

A combinatorial characterization of graph rigidity, or generic framework rigidity, in 3-dimensions, sought for nearly 150 years since Maxwell \cite{maxwell:equilibrium:1864}, would ideally provide a polynomial-time algorithm to decide if an input graph is independent in the 3-dimensional generic rigidity matroid $\mathcal{R}_3$ over the complete graph $K_n$.  
Even showing the existence of short  certificate of dependence (polynomial in the size of the input graph), i.e. placing the problem in the complexity class co-NP, would represent significant progress.  
For any matroid with an algebraic representation over any field of characteristic 0 or large enough finite field, the problem is in NP and in fact has a randomized polynomial-time (RP) algorithm \cite{demillo1978probabilistic,schwartz1980fast,zippel1979probabilistic}, which guarantees numerous short certificates for independence.  
A co-NP characterization would therefore place the problem in RP $\cap$ co-NP.   

A closely related matroid $\mathcal{M}_{K_5}$ is defined as (a) belonging to the class of matroids in which $K_5$-isomorphic graphs are circuits which includes all abstract 3-dimensional rigidity matroids, and  (b) ensuring independence of any graph that is independent in any matroid in this class.
Such a matroid is unique if it exists, and a co-NP characterization of dependence is based on the intuition that any $\mathcal{M}_{K_5}$-\emph{implied nonedge} $e$ of a graph $G$ - defined as belonging to the closure $CL_{\mathcal{M}_{K_5}}(G)$ but not in $G$- is contained within a $K_5$ in this closure.  
Clinch, Jackson and Tanigawa \cite{Clinch2022Abstract,Clinch2022Abstract2} showed that  Whiteley's cofactor matroid \cite{whiteley:Matroids:1996} from spline approximation theory is $\mathcal{M}_{K_5}$, i.e. it satisfies (a) and (b) above. 
It is this algebraic 
representation that further gives an RP $\cap$ co-NP characterization of independence in $\mathcal{M}_{K_5}$.  

\begin{figure}[htb]
    \centering
    \begin{subfigure}[t]{0.49\linewidth}
        \centering
        \includegraphics[width=0.5\linewidth]{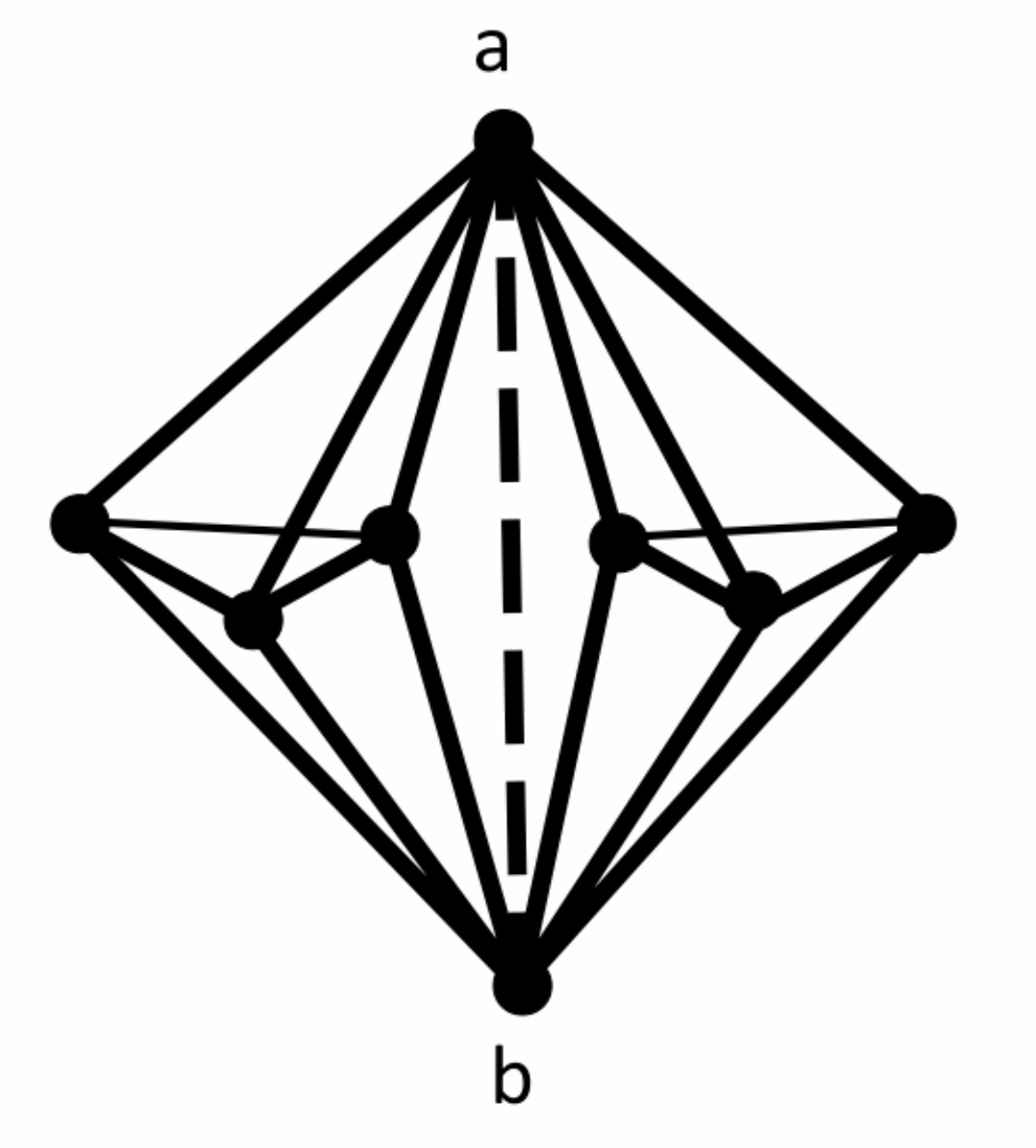}
    \end{subfigure}
    \begin{subfigure}[t]{0.49\linewidth}
        \centering
        \includegraphics[width=0.8\linewidth]{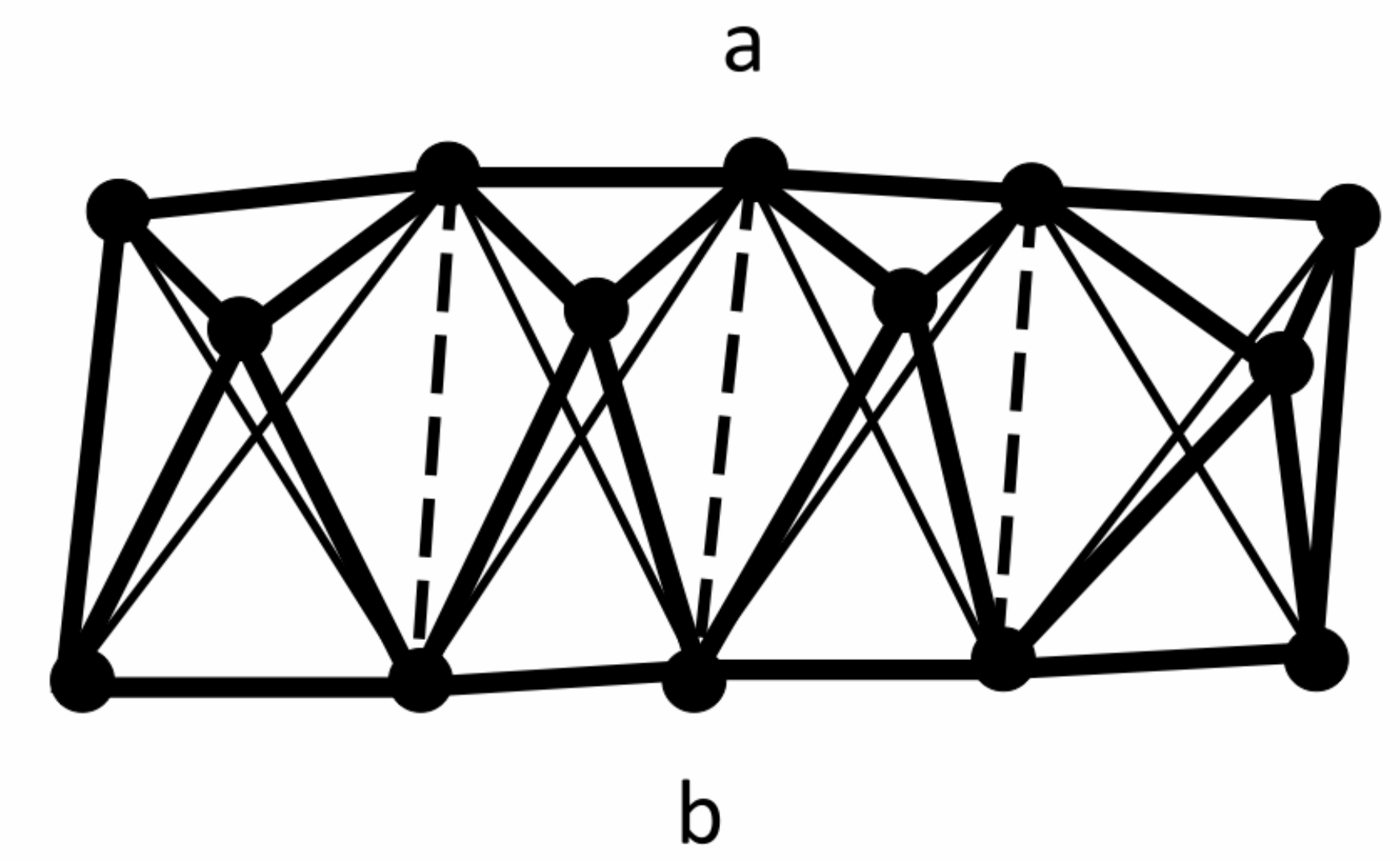}
    \end{subfigure}
    \caption{(Left) a double-banana and (Right) Crapo's hinge both with $(a,b)$ as an $\mathcal{R}_3$-implied nonedge.  
    See the flexible circuit discussion in Section \ref{sec:introduction}.}
    \label{fig:doubleBanana}
\end{figure}

Improving this characterization to a polynomial-time characterization would be a breakthrough, not least because of the maximality conjecture of Graver \cite{graver1991rigidity} which suggests that $\mathcal{R}_3$ and $\mathcal{M}_{K_5}$ are identical.  
An example sufficient to counter the maximality conjecture would be an $\mathcal{R}_3$-implied nonedge $e$ of a graph $G$ that is not contained in any $K_5$ in the closure $CL_{\mathcal{R}_3}(G)$.  
A necessary condition for such a counterexample is that any subgraph $G' \subseteq G$ -- for which   $G' \cup e$ is a circuit -- is flexible.  
However, this condition is far from sufficient.  
Both $\mathcal{M}_{K_5}$ and $\mathcal{R}_3$ have have implied nonedges arising only from flexible circuits such as the well-known double-banana and Crapo's hinge shown in Figure \ref{fig:doubleBanana}, and many others \cite{crapo:structuralRigidity:1979,grasegger2022flexible}.
 
In this paper, we illuminate further obstacles to settling the maximality conjecture by giving inductive 
constructions of a particularly intractable class of flexible circuits arising from independent graphs that contain implied nonedges but do not contain any rigid subgraphs on at least $5$ vertices.  
We call such subgraphs \emph{nucleations} for reasons explained below. These constructions apply to both the abovementioned matroids, and  some of our results  apply to any 3-dimensional abstract rigidity matroid, with their respective definitions of  nucleations and  implied nonedges. While pointing out  all such results is beyond the scope of this paper, we indicate a few.   

While none of the explicit examples we provide refutes the maximality conjecture, since their properties hold in both $\mathcal{M}_{K_5}$ and $\mathcal{R}_3$, the construction schemes, which generate a special class of flexible circuits, are valid even if the conjecture is false.  
In other words, the constructions can build on small (counter)examples  to generate larger ones, and several of the constructions can build on (counter)examples that separate any pair of abstract 3-dimensional rigidity matroids.  

It should be noted that so far even extremal potential counterexamples to the maximality conjecture have not been ruled out (see Figure \ref{fig:open}).  
For example, to the best of our knowledge, it is unknown whether there is a graph $G$ for which no $\mathcal{R}_3$-implied nonedge is  contained in any $K_5$ in the closure $CL_{\mathcal{R}_3}(G)$.   
It is even unknown whether there is a flexible circuit in $\mathcal{R}_3$ (with 6 or more vertices) that is closed (this is equivalent to the existence of a nucleation-free graph with a single implied nonedge).  
Therefore it is relevant that the constructions of this paper do not depend on the status of the maximality conjecture.  

Next we explain a second motivation for our constructions.  
Regardless whether the maximality conjecture is true or false, since graphs that are dependent in $\mathcal{M}_{K_5}$ are also dependent in $\mathcal{R}_3$, it is difficult to imagine a polynomial-time  algorithm that detects $\mathcal{R}_3$-dependence without also detecting $\mathcal{M}_{K_5}$-dependence.

Previous attempts at combinatorial algorithms \cite{hoffman2001decomposition,sitharam:zhou:tractableADG:2004} have in effect attempted the latter to approximately determine independence (i.e. output any known independence or dependence certificates) in part because they extend the corresponding polynomial-time algorithms for \cite{hoffmann1997finding,jacobs1997algorithm,lee2008pebble}   in 2-dimensions, or in general for sparsity matroids. These apply also to  special classes of graphs such as squared graphs or molecular graphs in 3-dimensions, for which  
the $\mathcal{R}_3$-independence   is equivalent to independence of a body-hinge, or panel-hinge structure \cite{katoh:tanigawa:proofMolecularConjecture:DCG:2011,tay1991linking,tay:rigidityMultigraphs-II:1989,whiteley1988union,crapo1982statics} in a corresponding matroid.  
These polynomial-time algorithms are a consequence of the equivalence of these rigidity matroids and appropriate sparsity matroids \cite{crapo1982statics,laman:rigidity:1970,nash1964decomposition,PG,bib:TayWhiteley85,tay:rigidityMultigraphs-I:1984,tay1991linking,tay:rigidityMultigraphs-II:1989,tay:proofLaman:1993,bib:counter3,tutte1961problem,white1987algebraic, whiteley:Matroids:1996}.  
For instance, not only is every implied nonedge $f$ of $G$ contained in some $K_4$ in the   closure $CL_{\mathcal{R}_2}(G)$ of $G$, the subgraph $G' \subseteq G$,  where $G'\cup f$ is a circuit, is in fact $\mathcal{R}_2$-rigid.  
Analogous statements hold in $\mathcal{R}_3$ for the above-mentioned special classes of graphs.

In contrast, because $\mathcal{M}_{K_5}$ and $\mathcal{R}_3$ are not sparsity matroids, they can have implied nonedges arising only from flexible circuits.  
Some previous combinatorial algorithms \cite{hoffman2001decomposition,sitharam:zhou:tractableADG:2004} have mitigated this issue by repeatedly alternating 2 steps: locating \emph{nucleations} (a term used in physical settings for structures that initiate or "nucleate" a process) and then replacing them with their closures, namely a clique, or a stand-in graph.  
Subgraph rigidity certificates are found e.g. using sparsity based methods for the above-mentioned special classes or by the application of rank bounds \cite{ChengSitharam2010, jackson:jordan:rank3dRigidity:egres-05-09:2005}.  
Nucleation-free graphs with implied nonedges, such as the families constructed in this paper, are the key obstacles that must be dealt with not only for improving such algorithms, but for obtaining polynomial-time characterizations of independence in $\mathcal{M}_{K_5}$ overall. 

These are the two primary motivations for the constructions given in this paper.  
See Section \ref{sec:contributions} and \ref{sec:conclusion} for further details.  
As corollary, we construct nucleation-free flexible circuits with arbitrarily many independent flexes.

\subsection{Organization}  
In Section \ref{sec:background} we provide a list of rigidity concepts used in the paper with pointers to literature. In Section \ref{sec:contributions}, we formally state our main theorems and other contributions.  
Section \ref{sec:rings-of-roofs} gives a warm up example, the "ring of butterflies," that provides the underlying intuition and, together with Section \ref{sec:induct},  gives  basic constructions.  
Section \ref{sec:split-and-glue} presents proofs of the main theorems that provide inductive constructions of independent, nucleation-free graphs with implied nonedges, along with conditions on the base graphs to which these constructions can be applied.  
Examples of base graphs are also provided.  
Section \ref{sec:dependent} gives corollaries that construct nucleation-free flexible circuits with arbitrarily many independent flexes, and the final Section \ref{sec:conclusion} concludes with a map of relevant open problems.

\subsection{Background}
\label{sec:background}
The paper assumes familiarity with the following concepts, constructions and results in  graph rigidity: \cite{graver:servatius:rigidityBook:1993,sitharam2018handbook}:
  
\begin{itemize}
\item
2-dimensional and 3-dimensional generic rigidity matroids $\mathcal{R}_2$ and $\mathcal{R}_3$ over $K_n$:
 generic frameworks,    rigidity matrices,   generic (infinitesimal) rigidity, flexes, independence,  circuits, dependence, and equilibrium self-stresses.
 
 \item Standard inductive constructions that preserve independence and/or rigidity including \emph{$k$-sums (without edge deletion), Henneberg-I/II extensions, and   vertex-splits}.  
 We call the neighborhood of the vertex added by a Henneberg-I/II extension the \emph{base vertex set} of this operation.  
 We call a vertex-split where $0 \leq k \leq d$ edges are added a \emph{$k$-vertex-split}.  
 
\item The  equivalence of $\mathcal{R}_2$ and the \emph{$(2,3)$-sparsity matroid}, and the unique \emph{maximal matroid} for which all graphs isomorphic to $K_4$ are circuits. 
\item \emph{Abstract 2- and 3- dimensional rigidity matroids} \cite{nguyen2010abstract}, standard inductive constructions including $k$-sums and Henneberg-I that preserve independence and/or rigidity. 
\item Whiteley's \emph{co-factor matroid} \cite{whiteley:Matroids:1996}, and its equivalence to the unique maximal matroid for which all graphs isomorphic to $K_5$ are circuits \cite{Clinch2022Abstract,Clinch2022Abstract2}, which we denote as $\mathcal{M}_{K_5}$. 
\end{itemize}

In addition, we define the following   for any 3-dimensional abstract rigidity matroid.  
A \emph{nucleation} of a graph $G$ is a rigid subgraph on at least five vertices.  
$G$ is \emph{nucleation-free} if it has no nucleation.  
A nonedge $f$ of $G$ is \emph{implied} if $G \cup f$ contains a circuit that includes $f$.  

\smallskip\noindent
\textbf{Notation and Conventions.}
The vertex and edge sets of $G$ are written as $V(G)$ and $E(G)$, respectively.  
For any subgraph $H$ of $G$ and any set $F$ of pairs of distinct vertices in $G$, we define the graph $H \cup F = (V(H), E(H) \cup F')$, where $F'$ is the subset of $F$ containing pairs of vertices whose endpoints are in $V(H)$. 
We write $H \cup F$ as $H \cup f$ when $F = \{f\}$.

We refer to concepts such as rigidity, independence, circuits, implied nonedges, without mentioning     $\mathcal{R}_3$, since all results in the paper also apply to $\mathcal{M}_{K_5}.$   
  All contributions listed in Section \ref{sec:contributions} except Theorems \ref{thm:ring_of_roofs}, \ref{thm:prev_ops} (iii, iv), \ref{thm:henneberg-ii_ring}, \ref{thm:drsg-ind}, and \ref{thm:starter}(i)   apply  to general 3-dimensional abstract rigidity matroids.

\section{Contributions}
\label{sec:contributions}

Theorem \ref{thm:ring_of_roofs}, below, gives a warm-up example of independent nucleation-free graphs with implied nonedges.  
We use this example to demonstrate a variety of proof methods.  

\begin{figure}[htb]
    \centering
    \begin{subfigure}[t]{0.49\linewidth}
        \centering
        \includegraphics[width=0.7\textwidth]{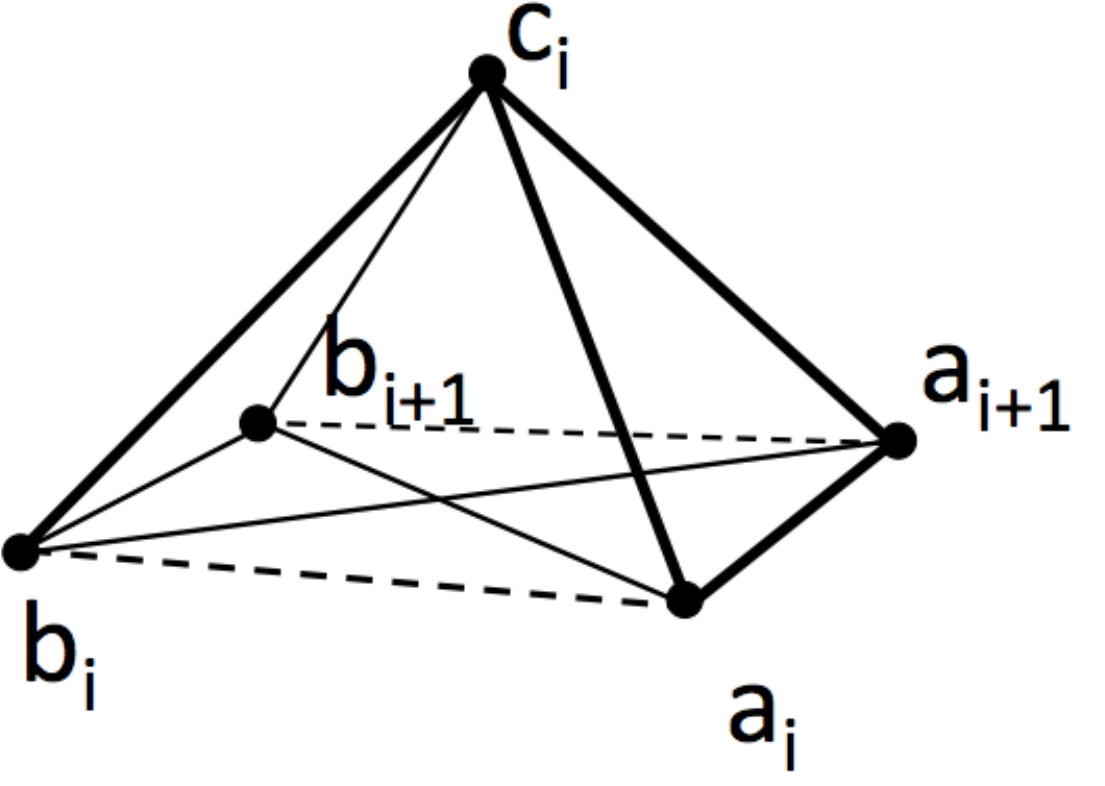}
    \end{subfigure}
    \begin{subfigure}[t]{0.49\linewidth}
        \centering
        \includegraphics[width=0.9\textwidth]{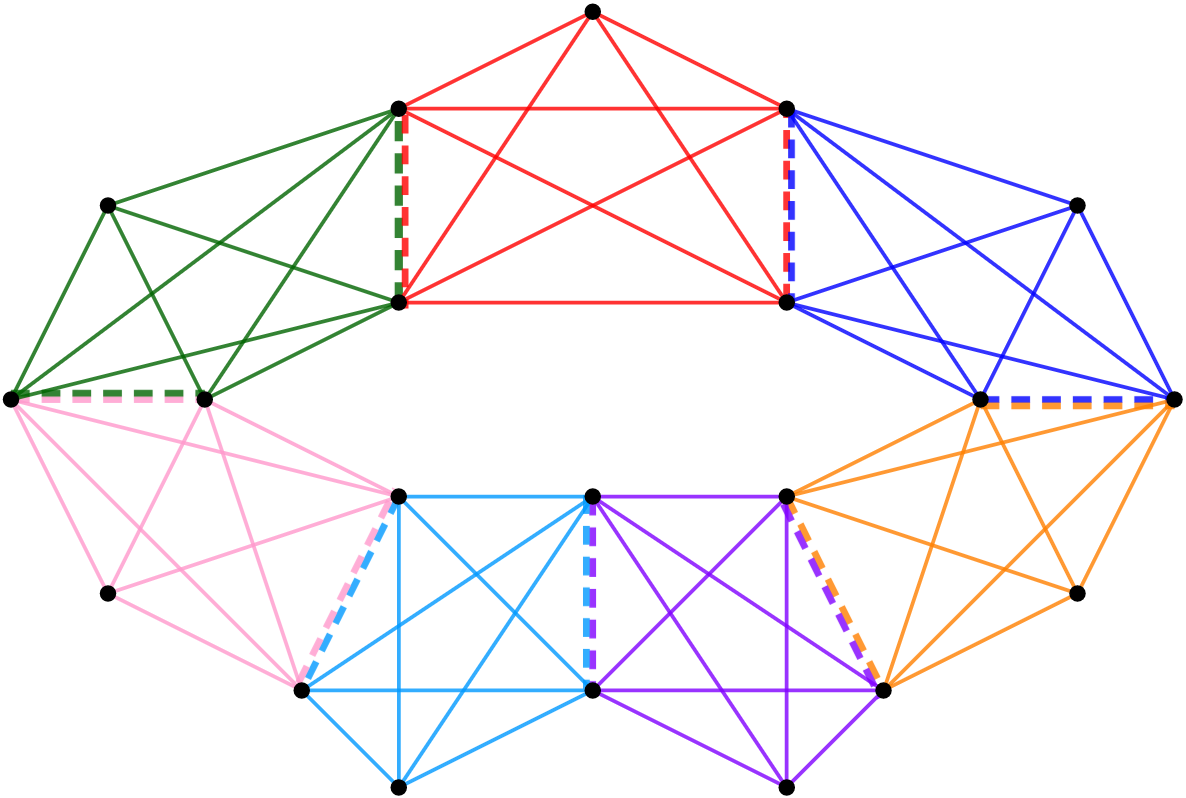}
    \end{subfigure}
    \caption{(Left) A butterfly graph and (Right) A ring of $7$ butterflies and a 2-thin cover whose clusters are colored and whose shared nonedges are depicted as parallel dashed line-segments of different colors.  
    See the definition of a ring of butterflies in Section \ref{sec:contributions} and the definition of a 2-thin cover in Section \ref{sec:rings-of-roofs}.}
    \label{fig:roof}
\end{figure}


\begin{definition}[Ring]
    A \emph{ring} $R_m = R(G_1,\dots,G_m)$ is a graph constructed from $m \geq 3$ vertex-disjoint graphs $G_1,\dots,G_m$, called \emph{links}, where each $G_i$ contains a set $X_i=\{a_i,b_i,c_i,d_i\}$ of \emph{hinge vertices}, by identifying $(a_1,b_1)$ with $(c_m,d_m)$ and every other $(c_i,d_i)$ with $(a_{i+1},b_{i+1})$.  
    The pairs $(a_i,b_i)$ and $(c_i,d_i)$ are called \emph{hinges} (of $G_i$ or $R_m$).  
\end{definition}

A \emph{butterfly} is a graph obtained from a complete graph $K_5$ by deleting two non-incident edges, and a \emph{ring of butterflies} is a ring whose links are all butterflies (see Figure \ref{fig:roof}).  


\begin{theorem}[Rings of butterflies have implied nonedges]
    \label{thm:ring_of_roofs}
    Each hinge in a ring of butterflies $R_m$ is implied.  
    Furthermore,
    for $m \geq 7$, $R_m$ is nucleation-free, 
    for $m \geq 6$, $R_m$ is independent with $m-6$ independent flexes, and  
    for $m \leq 5$, $R_m$ is dependent.
\end{theorem}

We give two proofs of Theorem \ref{thm:ring_of_roofs} in Section \ref{sec:rings-of-roofs} that utilize two techniques for showing that the nonedge hinges of a ring are implied.  
The \emph{flex-sign} technique relies only on the existence of certain generic frameworks of a ring, but has the serious disadvantage that the ring of butterflies is the only known example of a ring with such frameworks.  
On the other hand, the \emph{rank-sandwich} technique is used throughout this paper to show the existence of implied nonedges in both rings and more general graphs, but requires these graphs to be independent and to have a special $2$-thin covers, defined in Section \ref{sec:rings-of-roofs}.  

By Theorem \ref{thm:ring_of_roofs}, rings of butterflies are examples of arbitrarily large independent nucleation-free graphs with implied nonedges.  
Theorem \ref{thm:prev_ops}, below shows that standard inductive constructions preserve some or all of these properties.

\begin{theorem}[Constructions preserving independence, nucleation-freeness, and implied nonedges]
    \label{thm:prev_ops}
    Let $G$ and $H$ be independent and nucleation-free with sets $F_1$ and $F_2$ of implied nonedges, respectively.  
    Then,
    \begin{enumerate}[(i)]
        \item For any $k \in \{0,1,2\}$, a $k$-sum of $G$ and $H$ is independent and nucleation-free and $F_1 \cup F_2$ is a set of its implied nonedges.  
        This also holds for $k = 3$ if the base complete graphs are not contained in any $K_4$ subgraph of $G$ or $H$.  
        
        \item A Henneberg-I extension on $G$ yields an independent nucleation-free graph in which all nonedges in $F_1$ are implied if the base vertex set is not contained in any $K_4$ subgraph of $G$.
        
        \item A Henneberg-II extension on $G$ yields an independent nucleation-free graph $G'$ if the base vertex set does not induce a $K_4$ subgraph of $G$.  
        If the edge deleted by this operation is implied in $G'$, then so is each nonedge in $F_1$.  
        
        \item For any $k \in \{0,1,2\}$, a $k$-vertex split on $G$ yields an independent nucleation-free graph.  
    \end{enumerate}
\end{theorem}

\begin{remark}
    Even without the $K_4$ conditions,  Items (i)-(iii) preserve independence and implied nonedges; and  Item (i) preserves implied nonedges for all $k$ and independence for $k\le 4$.
\end{remark}

Theorem \ref{thm:prev_ops} is proved in Section \ref{sec:induct}.  

Theorem \ref{thm:henneberg-ii_ring}, stated below and proved in Section \ref{sec:induct}, provides an inductive construction using Henneberg-II extensions to take an independent ring with nucleations and possibly no implied nonedges and transform it into an independent nucleation-free ring whose hinges are implied nonedges.  

\begin{theorem}[Constructing nucleation-free rings with implied nonedges via Henneberg-II extensions]
    \label{thm:henneberg-ii_ring}
    Consider a ring $R'_m$, with $m \geq 7$, obtained from a ring $R_m=R(G_1,\dots,G_m)$, whose hinges are all edges, by performing a Henneberg-II extension on each link $G_i$ that uses its hinge vertices as the base vertex set and deletes a hinge edge.  
    If $R_m$ is independent, each $G_i$ is rigid, and the graph obtained from each $G_i$ by deleting its hinge edges is nucleation-free, then $R'_m$ is independent and nucleation-free, and each of its hinges is implied.  
\end{theorem}

\begin{remark}
    Theorem \ref{thm:henneberg-ii_ring} can be strengthened by replacing the rigidity condition on the links $G_i$, which we call \emph{seed graphs}  of the construction in Theorem \ref{thm:henneberg-ii_ring}, with a weaker seed graph condition:  if $G'_i$  is the resulting graph after the Henneberg-II extension  with  $e \in E(G_i) \setminus E(G'_i)$ being the deleted edge, then $e$ is an implied nonedge in $G'_i$.  
\end{remark}

 In order to state our main contributions,  we introduce the \emph{\SG/}  inductive construction. We then demonstrate that this operation preserves nucleation-freeness in a general case (Theorem \ref{thm:sg-nf}), independence in a special case (Theorem \ref{thm:drsg-ind}), and a particular set of implied nonedges in a slightly less general case (Theorem \ref{thm:sg-implied}).  
We define this inductive construction in three steps.  
\begin{figure}[htb]
    \centering
    \begin{subfigure}[t]{0.49\linewidth}
        \centering
        \includegraphics[width=0.7\textwidth]{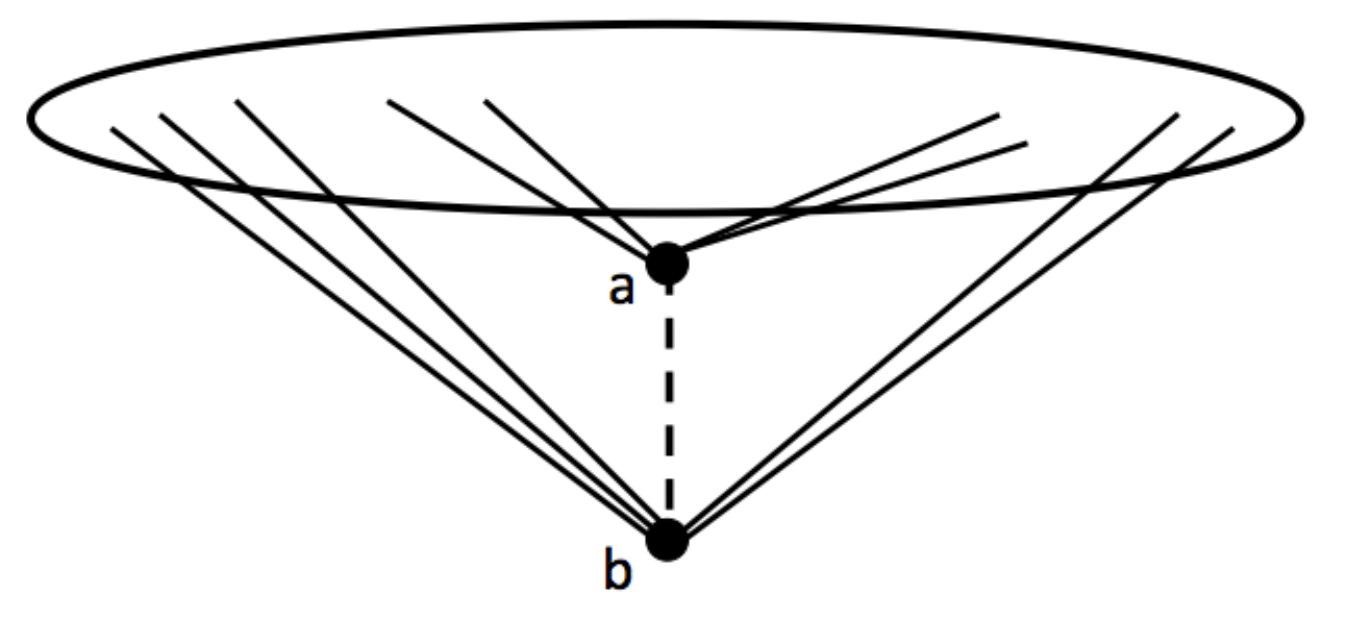}
    \end{subfigure}
    \begin{subfigure}[t]{0.49\linewidth}
        \centering
        \includegraphics[width=0.7\textwidth]{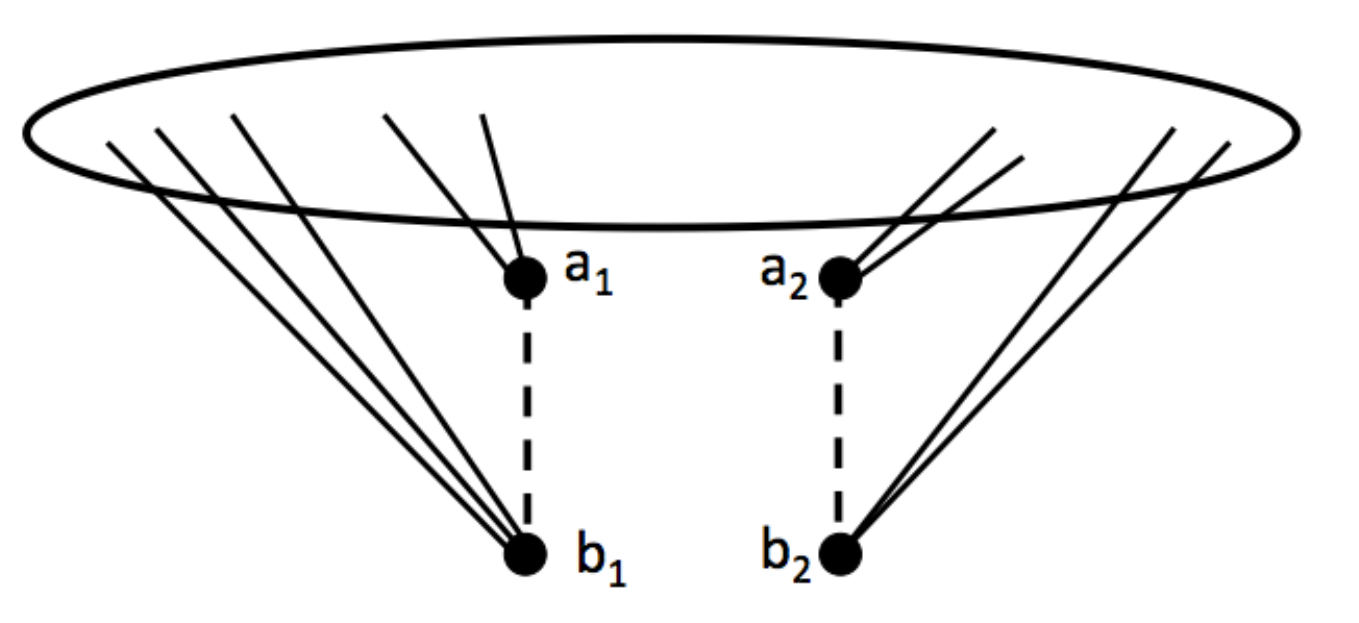}
    \end{subfigure}
    \caption{A nonedge-split operation defined in Section \ref{sec:contributions}.}
    \label{fig:splitting}
\end{figure}
\begin{figure}[htb]
    \centering
    \begin{subfigure}[t]{0.49\linewidth}
        \centering
        \begin{tikzpicture}[scale=0.75]
            \node[draw,circle,fill,inner sep=-1.5pt,label=left:{$a_1$}] (a1) at (0,0) {};

            \node[draw,circle,fill,inner sep=-1.5pt,label=right:{$b_1$}] (b1) at (4,0) {};

            \node[draw,circle,fill,inner sep=-1.5pt,label=left:{$v$}] (v) at (1,-2) {};

            \node[draw,circle,fill,inner sep=-1.5pt,label=right:{$u$}] (u) at (3,-2) {};

            \node[draw,circle,fill,inner sep=-1.5pt,label=left:{$a_2$}] (a2) at (0,-4) {};

            \node[draw,circle,fill,inner sep=-1.5pt,label=right:{$b_2$}] (b2) at (4,-4) {};

            \node[draw,circle,fill,inner sep=-1.5pt,label={[yshift=-0.55cm]$c$}] (c) at (2,-0.5) {};

            \node[draw,circle,fill,inner sep=-1.5pt,label={[xshift=0.03cm,yshift=0.05cm]$c'$}] (c') at (2,-3.5) {};

            \draw (a1) -- (v);
            \draw (a1) -- (u);

            \draw (b1) -- (v);
            \draw (b1) -- (u);

            \draw (a2) -- (v);
            \draw (a2) -- (u);

            \draw (b2) -- (v);
            \draw (b2) -- (u);

            \draw (c) -- (a1);
            \draw (c) -- (b1);
            \draw (c) -- (v);
            \draw (c) -- (u);

            \draw (c') -- (a2);
            \draw (c') -- (b2);
            \draw (c') -- (v);
            \draw (c') -- (u);

            \draw[dashed] (a1) -- (b1);
            \draw[dashed] (v) -- (u);
            \draw[dashed] (a2) -- (b2);
        \end{tikzpicture}
    \end{subfigure}
    \begin{subfigure}[t]{0.49\linewidth}
        \centering
        \includegraphics[width=0.9\textwidth]{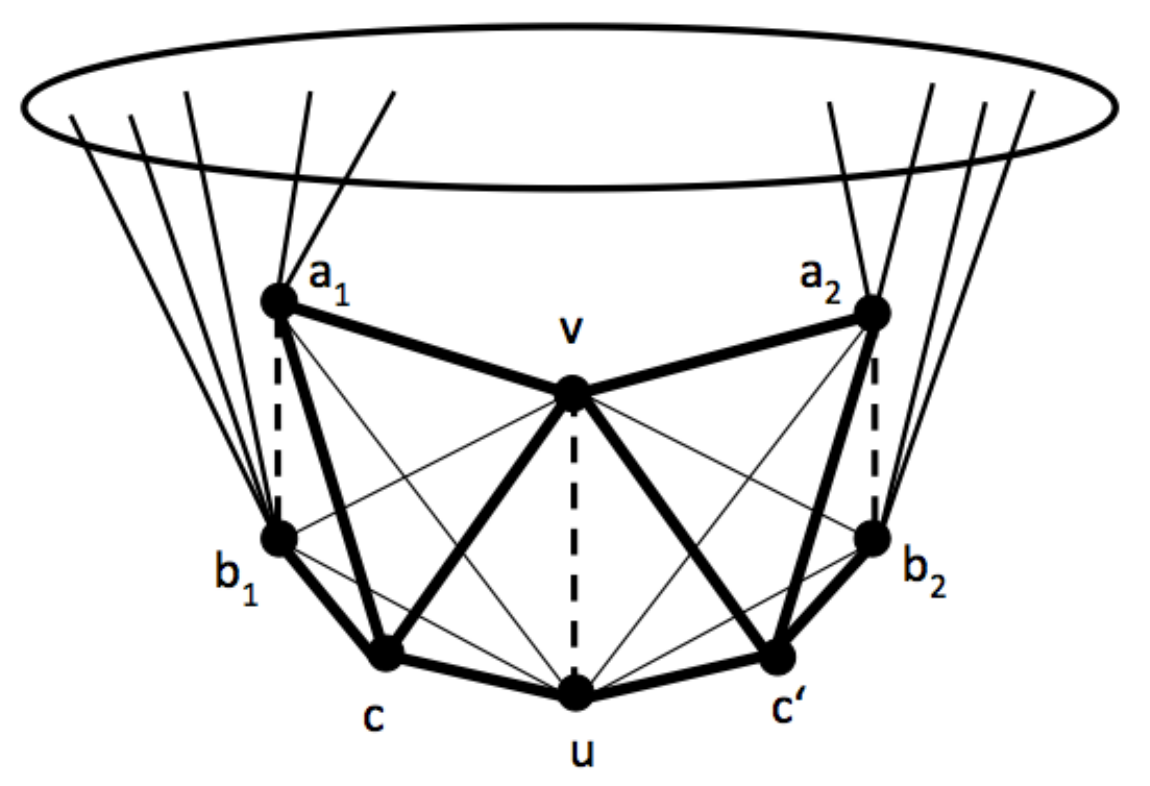}
    \end{subfigure}
    \caption{(Left) A double butterfly and (Right) a glue involving a double-butterfly.  
    See the glue and \SG/ inductive construction defined in Section \ref{sec:contributions} and the proof sketch of Theorem \ref{thm:drsg-ind} in Section \ref{sec:split-and-glue}.}
    \label{fig:twoRoofs}
\end{figure}
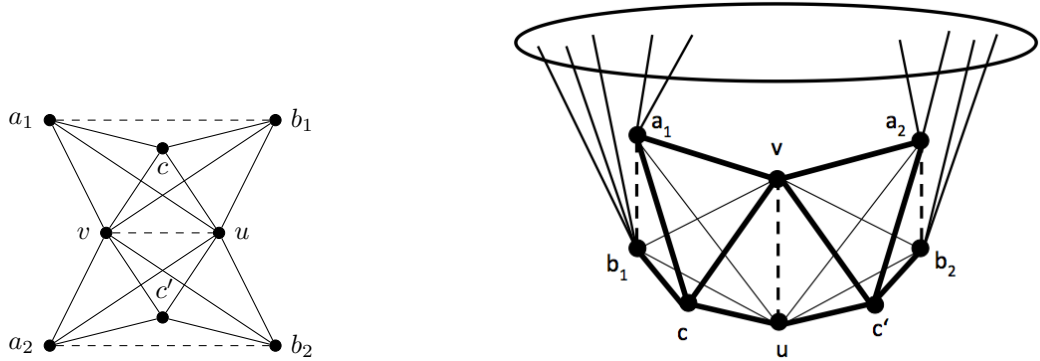
\begin{itemize}
    \item A \emph{nonedge-split} on a graph $G$ yields a \emph{split graph} $G^s$ by taking a nonedge $(a,b)$, called the \emph{split nonedge}, adding new vertices $a_1$ and $a_2$ whose neighborhoods form an arbitrary partition of the neighborhood of $a$ and new vertices $b_1$ and $b_2$ whose neighborhoods form an arbitrary partition of the neighborhood of $b$, and finally deleting $a$ and $b$.  
    See Figure \ref{fig:splitting}.
    
    \item A \emph{glue} on vertex-disjoint graphs $G$ and $H$ yields a graph $G:H$ by bijectively identifying some \emph{gluing vertices} of $G$ with some \emph{gluing vertices} of $H$.  
    See Figure \ref{fig:twoRoofs}.  
    
    \item A \emph{\SG/} on a \emph{base graph} $G$ and an \emph{ear} $H$ consists of a nonedge-split on $G$ yielding $G^s$ followed by a glue on $G^s$ and $H$ yielding the \emph{\SG/ graph} $G^s:H$ such that the gluing vertices of $G^s$ are a subset of those created by the nonedge-split.  
\end{itemize}
We also present two variants of this inductive construction.
\begin{itemize}
    \item A \emph{double-butterfly \SG/} is a \SG/ whose ear is a \emph{double-butterfly} as in Figure \ref{fig:twoRoofs}: two butterflies glued together on the nonedge $(u,v)$ whose gluing vertices are $a'_1$, $a'_2$, $b'_1$, and $b'_2$.  
    \item A \emph{safe \SG/} is a \SG/ with additional  conditions on  the constituent graphs and operations which, while fairly natural,  are too technical state here, but formally defined in Section \ref{sec:dra-implied}. In particular, these conditions imply that the base graph is independent and has an implied nonedge.  
\end{itemize}

\begin{remark}
    Although not proved here, we believe the graph obtained via a double-butterfly \SG/ on a graph $G$ is  not obtainable from $G$ via any sequence of the standard inductive constructions discussed in Theorem \ref{thm:prev_ops}.  
\end{remark}

\begin{remark}
    It may not always be possible to perform a safe \SG/ on a given base graph and ear;  however, it can be performed if natural conditions  (detailed in Section \ref{sec:dra-implied}) are satisfied by the  base graph  and ear (in particular, if the ear is a double butterfly).  
\end{remark}

\begin{theorem}[Nucleation-free preserving \SG/]
    \label{thm:sg-nf}
    If $G$ and $H$ are independent graphs (in fact, sparsity is sufficient), then the graph $G^s:H$ resulting from a \SG/ is nucleation-free provided $G^s$ and $H$ are nucleation-free.
\end{theorem}

We are now ready to state our  main contributions.
\begin{theorem}[Independence preserving double-butterfly \SG/]
    \label{thm:drsg-ind}
    A double-butterfly \SG/ whose base graph is independent yields an independent graph.  
\end{theorem}

\begin{theorem}[Implied nonedge preserving safe \SG/]
    \label{thm:sg-implied}
    If the graph $G^s:H$ resulting from a safe \SG/ is independent (a safe \SG/ ensures $G$ has an implied nonedge), then it has some implied nonedge.  
\end{theorem}

Theorems \ref{thm:sg-nf}, \ref{thm:drsg-ind}, and \ref{thm:sg-implied} are proved in Sections \ref{sec:dra-ind} - \ref{sec:dra-implied}.  
Corollary \ref{prop:cons1}, below, is a consequence of Theorems \ref{thm:sg-nf} and \ref{thm:sg-implied} while Corollary \ref{prop:cons2}, also below, follows from all three of these theorems.  

\begin{corollary}[Nucleation-free and implied nonedge preserving \SG/]
    \label{prop:cons1}
    If the graph $G^s:H$ resulting from a safe \SG/ is independent and the split graph $G^s$ and ear $H$ are nucleation-free, then $G^s:H$ is nucleation-free and has some implied nonedge.  
\end{corollary}

\begin{corollary}[Independence, nucleation-free, implied nonedge preserving \SG/]
    \label{prop:cons2}
    A safe \SG/ whose split graph is nucleation-free and whose ear is a double-butterfly yields an independent nucleation-free graph that has some implied nonedge.  
\end{corollary}

Theorem \ref{thm:starter}, stated below and proved in Section \ref{sec:safe_base}, concerns \emph{safe base graphs}, which are graphs that can be used as base graphs in safe \SG/s.  

\begin{theorem}[Constructing safe base graphs]
    \label{thm:starter}
    The following inductive constructions yield a safe base graph:
    \begin{enumerate}[(i)]
        \item a safe double-butterfly \SG/, 
        \item a $k$-sum on a safe base graph and an independent graph, for any $k \geq 0$, and 
        \item a Henneberg-I extension on a safe base graph.
    \end{enumerate}
\end{theorem}

\begin{corollary}[Inductive constructions yielding safe base graphs]
    Let $G$ be a nucleation-free safe base graph and let $G'$ be the resulting graph 
  after any of the  inductive constructions (i), (ii) or (iii)  in Theorem \ref{thm:starter} is applied. 
 Then $G'$ is a nucleation-free safe base graph after Operation (i); or  after Operations  (ii) or (iii)  provided they additionally satisfy the conditions of Theorem \ref{thm:prev_ops} (i) and (ii) respectively. 
\end{corollary}

\begin{proof}
    The corollary follows immediately from Theorems \ref{thm:prev_ops}, \ref{thm:sg-nf}, and \ref{thm:starter}.
\end{proof}

\begin{remark}
    In the above corollary, if construction (i) is applied, we can weaken the condition that $G$ is nucleation-free  to the condition that the split graph $G^s$ is nucleation-free.
\end{remark}

Finally, Theorem \ref{thm:dependent}, stated below and proved in Section \ref{sec:dependent}, gives a straightforward method to combine nucleation-free graphs with implied nonedges to construct a larger nucleation-free graph with implied nonedges that is dependent.  
For any subgraph $H$ of a graph $G$ and any set $F$ of pairs of distinct vertices in $G$, we define the graph $H \cup F = (V(H), E(H) \cup F')$, where $F'$ is the subset of $F$ containing all pairs of vertices whose endpoints are in $V(H)$.  
We write $H \cup F$ as $H \cup f$ when $F = \{f\}$.  

\begin{theorem}[Constructing dependent nucleation-free graphs w/ implied nonedges]
    \label{thm:dependent}
    Let $G_1$ and $G_2$ be edge-disjoint graphs that share the endpoints of a nonedge $f$.  
    Then, $G_1 \cup G_2$ is
    \begin{enumerate}[(i)]
        \item nucleation-free if $G_1$ and $G_2$ are nucleation-free and vertex-disjoint except for the endpoints of $f$, 
        \item dependent if $f$ is implied in both $G_1$ and $G_2$, and
        \item a circuit if $G_1$ and $G_2$ are vertex-disjoint except for the endpoints of $f$ and $G_1 \cup f$ and $G_2 \cup f$ are circuits.  
    \end{enumerate}
    Additionally, $f$ is implied in $G_1 \cup G_2$ in Cases (ii) and (iii).  
\end{theorem}

\section{Warm-up: rings of butterflies}
\label{sec:rings-of-roofs}

Here we prove Theorem \ref{thm:ring_of_roofs}.  
Lemma \ref{lem:ror-implied}, below, shows that each hinge in a ring of butterflies is implied, and we give two proofs using the flex-sign and the rank-sandwich techniques discussed in Section \ref{sec:contributions}.  

\begin{lemma}[Ring of butterflies has implied hinges]
    \label{lem:ror-implied}
    Each hinge of a ring of butterflies is implied.
\end{lemma}

It is easy to check that the ring is rigid when it has at most $6$ links, and hence the lemma is immediate in this case and only interesting when the ring has at least $7$ links.  
The first proof uses the flex-sign technique.  
Recall that a flex of a framework $G(p)$ is a vector $u = (u_1,\dots,u_n)$ in the right-kernel of its rigidity matrix, which can be viewed as an assignment of infinitesimal velocities $u_i$ to the points $p_i$.  

\begin{definition}[Expansive and contractive flexes]
    Let $G(p)$ be a framework with a flex $u$ an edge $(i,j)$ such that wlog $p_i$ is the origin and $p_j$ lies on the positive $x$-axis.  
    Also, let $x(p_i)$, $x(p_j)$, $x(u_i)$, and $x_(u_j)$ be the $x$-coordinates of these vectors.  
    Right-multiplying the rigidity matrix of $G(p)$ by $u$ yields  
    $$x(u_i)(x(p_i)-x(p_j)) - x(u_j)(x(p_i)-x(p_j))$$
    corresponding to the row for $(i,j)$.  
    We say $u$ \emph{expands} $(p_i,p_j)$ when this term is positive, and $u$ \emph{contracts} $(p_i,p_j)$ when this term is negative.  
    Since $x(p_i) = 0$, expansion and contraction occur exactly when $x(u_i)<x(u_j)$ and $x(u_i)>x(u_j)$, respectively.  
    We say $u$ expands (resp. contracts) a nonedge $f$ of $G$ if it expands (resp. contracts) the edge $f$ in $G \cup f$.  
    A triple $(G(p),f_1,f_2)$, where $f_1$ and $f_2$ are nonedges, is (i) \emph{positive} if any flex of $G(p)$ expands (resp. contracts) $f_1$ if and only if it expands (resp. contracts) $f_2$ and (ii) \emph{negative} if any flex of $G(p)$ expands (resp. contracts) $f_1$ if and only if it contracts (resp. expands) $f_2$.  
\end{definition}

For any graph $G$ whose vertex set is $\{1,\dots,n\}$, any subgraph $H$ whose vertex set is $\{v_1,\dots,v_m\}$, and any framework $G(p)$, let $H(p)$ be the framework containing $H$ and the point $(p_{v_1},\dots,p_{v_m})$.  

\begin{figure}[htb]
    \centering
    \includegraphics[width=0.6\textwidth]{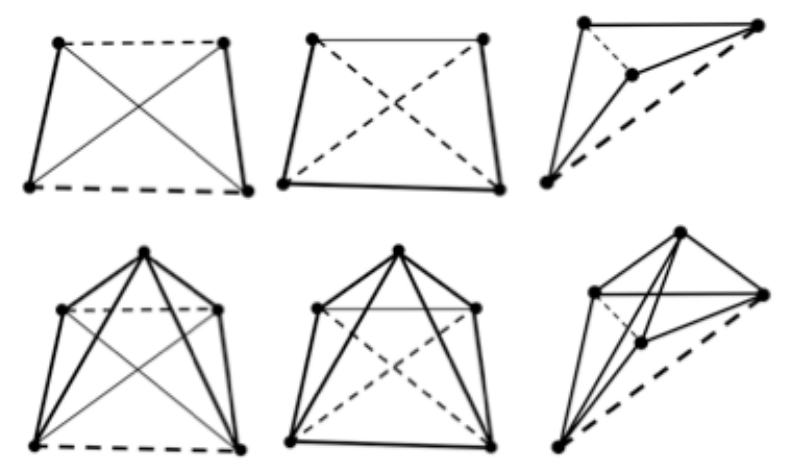}
    \caption{The top shows three frameworks of $K_4$ minus two edges in $\mathbb{R}^2$.  
    Together with the dashed nonedges, the first two are negative \cite{connelly:rigidityAndEnergy:1982} while the third is positive \cite{streinu:pseudoTriang:dcg:2005}.  
    It was shown in these papers that coning these frameworks, i.e. adding one new vertex connected to all other vertices, preserved these properties.  
    The bottom three frameworks from left to right are called a crossing, convex, and pseudo-triangular butterfly.  
    See the proof of Lemma \ref{lem:ror-implied} via flex-sign technique in Section \ref{sec:rings-of-roofs}.}
    \label{fig:convexPseudoPTbutterflies}
\end{figure}

\begin{proof}[Proof of Lemma \ref{lem:ror-implied} via flex-sign technique]
    Consider any hinge $f$ of a ring of butterflies $R_m = R(G_1,\dots,G_m)$.  
    Observe that there exists a framework $R_m(p)$ where the points of $G_1(p)$ are in strict convex position, i.e. either  a \emph{crossing butterfly} or a \emph{convex butterfly} in Figure \ref{fig:convexPseudoPTbutterflies}, while for all $i \geq 2$, the points of $G_i(p)$ are in strict concave position, i.e., a \emph{pseudo-triangular butterfly}) in Figure \ref{fig:convexPseudoPTbutterflies}.  
    Furthermore, there exists some neighborhood of $R_m(p)$ where these properties hold, and so we can assume $R_m(p)$ is generic.  
    Letting $f_i$ and $f'_i$ be the hinges of $G_i$, it was shown in \cite{connelly:rigidityAndEnergy:1982} that $(G_1(p),f_1,f'_1)$ is negative while it was shown in \cite{streinu:pseudoTriang:dcg:2005} that $(G_i(p),f_i,f'_i)$ is positive for all $i \geq 2$.  
    If $f$ is not implied, then $R_m(p)$ is flexible and, wlog, one of its flexes $u$ expands $f$.  
    However, following the signs of triples $(G_i(p),f_i,f'_i)$ around the ring shows that $u$ contracts $f$, which is a contradiction.  
    Thus, $f$ is implied,  proving the lemma.  
\end{proof}

\begin{figure}[htb]
    \centering
    \begin{subfigure}[t]{0.49\linewidth}
        \centering
        \includegraphics[width=0.7\textwidth]{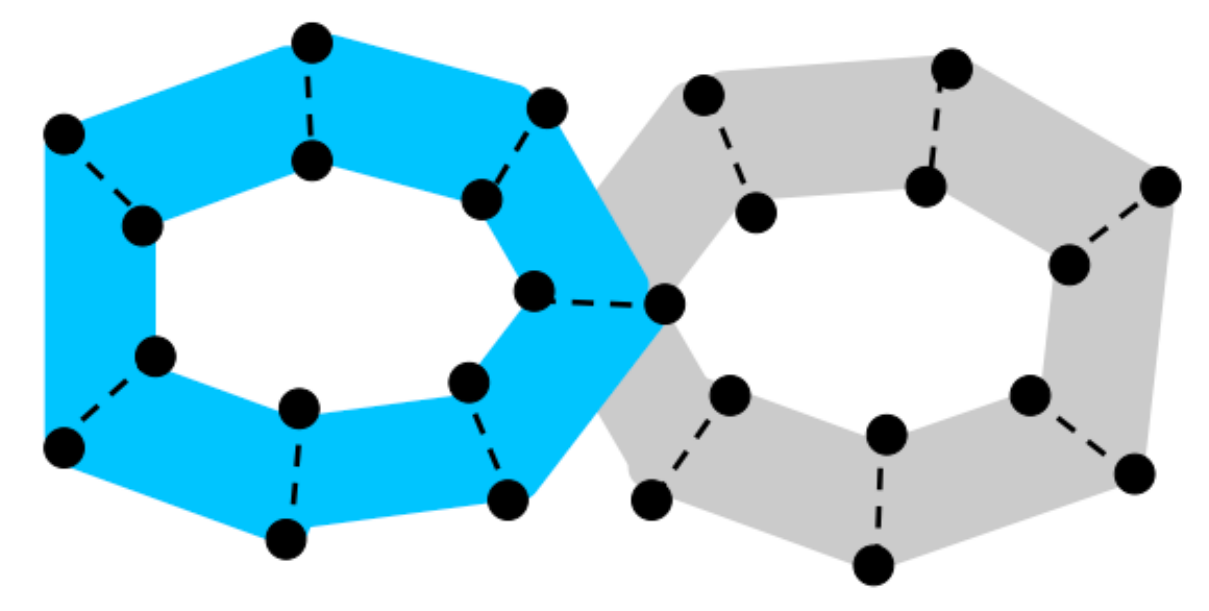}
    \end{subfigure}
    \begin{subfigure}[t]{0.49\linewidth}
        \centering
        \includegraphics[width=0.7\textwidth]{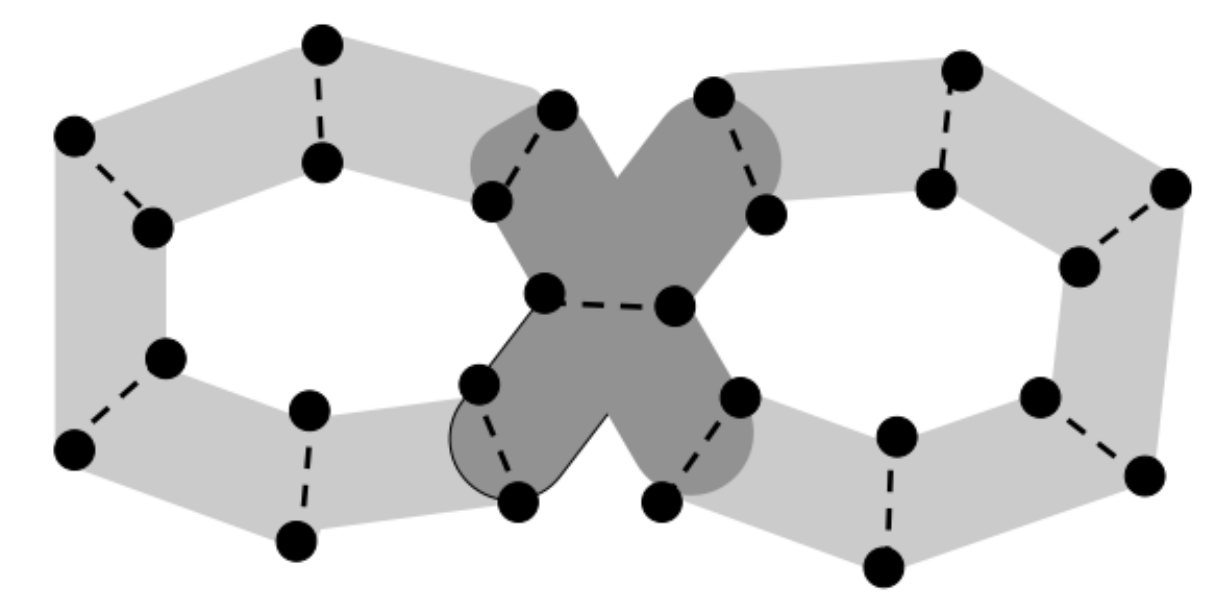}
    \end{subfigure}
    \caption{On the left and right is a graph $G$ consisting of two rings sharing a hinge edge.  
    The left illustrates a $2$-thin cover of $G$ consisting of two clusters, which are the vertex sets of the two rings.  
    The right illustrates another cover of $G$ where the vertex set in the dark-gray region is one cluster and the vertex sets of the links in the light-gray region are the other clusters.  
    Yet another cover of $G$ contains the vertex set of each link as a cluster.  
    See the definition of a $2$-thin cover in Section \ref{sec:contributions}.}
    \label{fig:twoCovers}
\end{figure}

The second proof uses the rank-sandwich technique, which shows that, for a graph $G$ and nonedge $f$, $|E(G)|$ is a lower-bound for $rank(G)$ and an upper-bound for  $rank(G \cup f)$.  
Hence, $rank(G) = rank(G \cup f)$,  hence  nonedge $f$ is implied in $G$.  
The lower-bound follows by showing that $G$ is independent.  
The upper-bound is demonstrated via the following result from \cite{jackson:jordan:rank3dRigidity:egres-05-09:2005}.  
A \emph{$2$-thin cover} of a graph $G$ is a collection $\mathcal{X} = \{X_1,\dots,X_k\}$ of subsets of $V(G)$, called \emph{clusters}, such that $|X_i| \geq 2$ for all $i \geq 1$, $|X_i \cap X_j| \leq 2$ and $X_i \not\subset X_j$ for all $i \neq j$, and both endpoints of each edge in $E(G)$ are contained in some $X_i$.  
See Figure \ref{fig:twoCovers} for examples.  
The \emph{shared set} $S(\mathcal{X})$ contains all pairs $(u,v)$ such that $\{u,v\} = X_i \cap X_j$ for some $i \neq j$.  
We call the 2-thin cover $\mathcal{X}$  and the shared set $S(\mathcal{X})$ \emph{independent} if the graph   $(V(S(\mathcal{X})), S(\mathcal{X}))$ is independent.  
For any edge or nonedge $e$ of $G$ and any collection $\mathcal{Y}$ of subsets of $V(G)$, $d(\mathcal{Y},e)$ is the number of sets in $\mathcal{Y}$ that contain $e$.  
For any subgraph $H$ of $G$ and any set $F$ of pairs of distinct vertices in $G$, we sometimes we write $H \cup F$ as $H_F$.  
For any subset $U \subseteq V(G)$, let $G[U]$ be the subgraph of $G$ induced by $U$.  
The \emph{inclusion-exclusion rank count} for the pair $(G,\mathcal{X})$ is $IE(G,\mathcal{X}) = \sum_{X_i \in \mathcal{X}} rank(G_{S(\mathcal{X})}[X_i]) - \sum_{e \in S(\mathcal{X})} (d(\mathcal{X},e) - 1)$.  

\begin{lemma}[\cite{jackson:jordan:rank3dRigidity:egres-05-09:2005}, Lemma 3.5.]
    \label{thm:2-thin-rank}
    If a graph $G$ has an independent $2$-thin cover $\mathcal{X} = \{X_1,\dots,X_k\}$, then $rank(G) \leq IE(G,\mathcal{X}).$
\end{lemma}

We provide a proof here as a useful exercise and for completeness.

\begin{proof}
    Since $S(\mathcal{X})$ is independent, there exists a maximal independent subgraph $H$ of $G_{S(\mathcal{X})}$ that contains $S(\mathcal{X})$. Clearly $|E(H)| = rank(G_{S(\mathcal{X})})$.  
    Observe that\\
    $rank(G) \leq rank(G_{S(\mathcal{X})})$ (by definition of $G_{S(\mathcal{X})}$)  \\
    $=|E(H)| = \sum_{X_i \in \mathcal{X}} |E(H[X_i])| - \sum_{e \in S(\mathcal{X})} (d(\mathcal{X},e) - 1) $ \\
    $\le IE(G,\mathcal{X}),$
    since $|E(H[X_i])| \leq rank(G_{S(\mathcal{X})}[X_i])$.   
\end{proof}



\begin{lemma}[Sufficiently large ring of butterflies independent]
    \label{lem:ror-ind}
    A ring of butterflies with at least $6$ links is independent.  
\end{lemma}

\begin{proof}
    Observe that a ring of butterflies $R_m$ with $m \geq 6$ can be obtained from a ring $R'_m$ of $K_4$'s by performing a Henneberg-II extension on each link.  
    $R'_m$ is independent, by \cite[Theorem 1.3]{monks20255regulargraphs3dimensionalrigidity}, and Henneberg-II extensions preserve independence, and thus $R_m$ is independent.  
\end{proof}

Lemma \ref{lem:ror-ind} is claimed without proof in \cite[Example 2]{jackson:jordan:rank3dRigidity:egres-05-09:2005}.  
We are unaware of an alternative proof that does not rely on the independence of a ring $R'_m$ of $K_4$'s with $m \geq 6$.  
Although we achieve this using a recent result from \cite{monks20255regulargraphs3dimensionalrigidity}, it can also be shown using the theory of body-hinge structures \cite{tay:rigidityMultigraphs-II:1989,tay1991linking,whiteley1988union} or panel-hinge structures (special position body-hinge structures) \cite{crapo1982statics} as follows.  
A bar-joint framework of $R'_m(p)$ in which the four points of each $K_4$ are coplanar can be regarded as a panel-hinge framework $R'_m(q)$, and then \cite[Proposition 3.4]{crapo1982statics} shows that the space of flexes of $R'_m(q)$ has dimension $m - 6$.  
This quantity lower-bounds the dimension of the space of flexes of $R'_m(p)$, and is in-fact equal to it since each $K_4$ is rigid.  
Therefore, the rank of $R'_m(p)$ is $3|V(R'_m)| - 6 - (m - 6) = 5m$.  
We also have $|E(R'_m)| = 5m$.  
Combining these equalities shows that $R'_m(p)$ is independent.  
Finally, since $R'_m(p)$ is not generic yet independent, $R'_m$ is independent.  

\begin{proof}[Proof of Lemma \ref{lem:ror-implied} via rank-sandwich technique]
    Consider any ring of butterflies $R_m$.  
    When $m \leq 5$, the lemma follows from the fact that $R_m$ is rigid, as discussed above.  
    When $m \geq 6$, consider any hinge $f$ of $R_m$.  
    Let $\mathcal{X}$ be the $2$-thin cover of $R_m$ whose clusters are the vertex sets of the links of $R_m$, which is also a $2$-thin cover of $R_m \cup f$.  
    Since no two pairs in the shared set $S(\mathcal{X})$ share a vertex, $\mathcal{X}$ is clearly independent, and so Lemma \ref{thm:2-thin-rank} shows that $rank(R_m \cup f) \leq IE(R_m,\mathcal{X}) = 8m$.  
    We also get $rank(R_m) = |E(R_m)|$ from Lemma \ref{lem:ror-ind}.
    Since $|E(R_m)| = 8m$ and $|E(R_m)| \leq rank(R_m \cup f)$, we see that $rank(R_m) = rank(R_m \cup f)$, and thus $f$ is implied in $R_m$.  
\end{proof}

In the above proof, we can alternatively obtain $rank(R_m \cup f) \leq 8m$ using a panel-hinge argument similar to the one following Lemma \ref{lem:ror-ind}.  
Let $R'_m = R_m \cup F$, where $F$ is the set of all hinges of $R_m$.  
Replacing $K_4$ with $K_5$ in the above panel-hinge argument shows that the rank of $R'_m$ is at most $3|V(R'_m)| - 6 - (m - 6) = 8m$.  
Thus, we get $rank(R_m \cup f) \leq 8m$.

We now prove Theorem \ref{thm:ring_of_roofs}.   

\begin{proof}[Proof of Theorem \ref{thm:ring_of_roofs}]
    Consider any ring of butterflies $R_m$.  
    For $m \geq 7$, it is easy to see that $R_m$ is nucleation-free by inspecting each of its subgraphs.  
    For $m \geq 6$, Lemma \ref{lem:ror-ind} shows that $R_m$ is independent, and hence $rank(R_m) = |E(R_m)|$.  
    Furthermore, $3|V(R_m)| - 6 - |E(R_m)|$ is the number of independent flexes $R_m$ has, which counting shows is $m - 6$.  
    For $m \leq 5$, $R_m$ is dependent since $|E(R_m)| > 3|V(R_m)| - 6$.  
    Lastly, each hinge of $R_m$ is implied for $m \geq 7$ by Lemma \ref{lem:ror-implied}.  
    For $m \leq 6$, it suffices to show that $R_m$ is rigid.  
    From the panel-hinge argument following Lemma \ref{lem:ror-ind}, we see that a ring $R'_m$ of $K_4$'s is rigid if and only if $m \leq 6$.  
    Thus, since $R_m$ can be obtained from $R'_m$ via Henneberg-II extensions, which preserve rigidity, $R_m$ is rigid.  
\end{proof}

\begin{remark}
    Figure \ref{fig:nontrivialSG} in Section \ref{sec:dra-implied} shows an independent nucleation-free graph with implies nonedges that is even smaller than a ring of butterflies with $7$ links, and the techniques in this section can be used to verify these properties.  
\end{remark}

\section{Constructions preserving independence, nucleation-freeness, and implied nonedges}
\label{sec:induct}

Here we prove Theorems \ref{thm:prev_ops} and \ref{thm:henneberg-ii_ring}.  
As stated in the introduction, many of our results hold for abstract $dF$-rigidity matroids, using  tools such as the following Lemma \ref{lem:abs}.

\begin{lemma}[$k$-sums preserve independence in abstract rigidity matroids]
    \label{lem:abs}
    Let $G$ and $H$ be bases for abstract $d$-rigidity matroids on $K_{V(G)}$ and $K_{V(H)}$, respectively.  
    If $E(G) \cap E(H)$ is a basis for the abstract $d$-rigidity matroid on $K_{V(G) \cap V(H)}$, then $G \cup H$ is a basis for the abstract $d$-rigidity matroid on $K_{V(G) \cup V(H)}$.  
    Consequently, for any $k \leq d$, $k$-sums preserve independence in any abstract $d$-rigidity matroid.  
\end{lemma}

\begin{proof}
    Since $0$-extensions construct bases of abstract $d$-rigidity matroids \cite[Lemma 2.1]{nguyen2010abstract}, they can be used to extend 
    $E(H)\cap E(G)$  to a basis of $K_{V(H)}$ which also extends $G$ to a basis $G'$ of  $K_{V(H) \cup V(G)}$. I.e.,   the subgraph of $G'$ induced by $V(H)$ is a basis of $K_{V(H)}$.   
    We  obtain $G \cup H$ from $G'$ by a series of exchanges: add an edge of $H$ to $G'$ to get a unique circuit $C$ whose vertex set is guaranteed to be a subset of $V(H)$ and delete an edge in $C$ that is not an edge in $H$ to get another basis of $K_{V(G) \cup V(H)}$.  
    If at some point the circuit $C$ contains only edges of $H$, then $H$ is dependent, which is a contradiction.  
    Therefore, $G'$ is a basis of $K_{V(G) \cup V(H)}$.  
\end{proof}

\begin{proof}[Proof of Theorem \ref{thm:prev_ops}]
    For (i), independence of $G'$ follows immediately from Lemma \ref{lem:abs}.  
    Next, for all $k$, since for any nonedge $f \in F_1 \cup F_2$, any subgraph of $G \cup f$ or $H \cup f$ that is a circuit is also a subgraph of $G' \cup f$, $f$ is implied in $G'$.  
    Lastly, if $G'$ has a nucleation $N$, then since $G$ and $H$ are nucleation-free, the base complete graphs $C$ and $C'$ for the $k$-sum are $K_3$'s and $N$ is the union of a $K_4$ subgraph of $G$ that contains $C$ and a $K_4$ subgraph of $H$ that contains $C'$.  
    However, by assumption, either $C$ or $C'$ is not contained in any $K_4$ subgraph of $G$ or $H$, and so $G'$ is nucleation-free.  

    For (ii), let $G'$ be the graph resulting from a Henneberg-I extension on $G$ such that the base vertex set is not contained in any $K_4$ subgraph of $G$.  
    A similar argument shows that $G'$ is independent and each nonedge in $F_1$ is implied in $G'$.  
    Our assumption on $G$ and the base vertex set $W$ of this operation implies that any subgraph $H$ of $G$ that properly contains $W$ is independent and flexible.  
    Hence, we have $rank(H) = |E(H)| < 3|V(H)| - 6$.  
    Since Henneberg-I extensions preserve the number of independent flexes a graph has, the subgraph of $G'$ obtained via a Henneberg-I extension on $H$ is independent and flexible.  
    Using the above fact, we see that any subgraph of $G'$ on at least $5$ vertices is flexible, and so $G'$ is nucleation-free.  
    
    For (iii) and (iv), a similar argument shows that a Henneberg-II extension or a $k$-vertex split, for any $k \in \{0,1,2\}$, on $G$ yields an independent nucleation-free graph.  
    Finally, in the case of a Henneberg-II extension, we argue that each nonedge $f \in F_1$ is implied in $G'$ if the edge $e$ deleted by this operation is implied in $G'$.  
    If the circuit of $G \cup f$ containing $f$ is a subgraph of $G' \cup f$, then this is immediate.  
    Otherwise, the existence of a circuit in $G' \cup e$ that contains $e$ implies the existence of a circuit in $G' \cup f$ that contains $f$.  
    This completes the proof.  
\end{proof}

The proof of Theorem \ref{thm:henneberg-ii_ring} requires Lemmas \ref{lem:7-link-flexible} and \ref{7-link-nucleation-free}, below.  

\begin{lemma}[Link-spanning subgraphs of sufficiently large rings are flexible]
    \label{lem:7-link-flexible}
    For any ring with at least $7$ links, any of its subgraphs that is not contained in any link and not separated by any hinge is flexible.  
\end{lemma}

\begin{proof}
    Let $R_m = R(G_1,\dots,G_m)$ be a ring, with $m \geq 7$, and $T_m = T(H_1,\dots,H_m)$ be any of its subgraphs that is not contained in any link and not separated by any hinge, where each $H_i$ is the vertex-maximal induced subgraph of $G_i$ contained in $T_m$.  
    Also, let $V$ and $X_i$ be the vertex sets of $T_m$ and $H_i$, respectively.  
    Then, $\mathcal{X} = \{X_1,\dots,X_m\}$ is clearly an independent $2$-thin cover of $T_m$.  
    Let $T'_m$ be obtained from $T_m$ by adding as an edge each pair in $S(\mathcal{X})$ whose vertices are both contained in $T_m$.  
    Also, let $P(\mathcal{X})$ contain the set $X_i \cap X_j$ for all $i \neq j$, and let $s$ and $t$ be the number of sets in $P(\mathcal{X})$ of size one and two, respectively.  
    Note that $s + t = m$, $\sum_{e \in S(\mathcal{X})} (d(\mathcal{X},e)-1) = t$, $rank(T'_m[X_i]) \leq 3|X_i| - 6$, and $\sum_{X_i \in \mathcal{X}} |X_i| = |V| + s + 2t$.  
    Therefore, applying Lemma \ref{thm:2-thin-rank} gives $rank(T_m) \leq 3|V| - 3s - t < 3|V| - 6$, and so $T_m$ is flexible.  
\end{proof}

\begin{lemma}[Sufficiently large ring with nucleation-free links is nucleation-free]
    \label{7-link-nucleation-free}
    A ring with at least $7$ links is nucleation-free if each of its links is nucleation-free.
\end{lemma}

\begin{proof}
    Consider any ring whose links are all nucleation-free and that has at least $7$ links.  
    Any subgraph of this ring on at least $5$ vertices either is contained in a link, is separated by a hinge, or has neither of the previous properties.  
    The subgraph is flexible in the first case since each link is nucleation-free, in the second case clearly, and in the third case by Lemma \ref{lem:7-link-flexible}.  
    Thus, the ring is nucleation-free.  
\end{proof}

\begin{proof}[Proof of Theorem \ref{thm:henneberg-ii_ring}]
    $R'_m$ is independent since $R_m$ is independent and Henneberg-II extensions preserve independence.  
    To show that $R'_m$ is nucleation-free, we prove that any one of its links $G'_i$ is nucleation-free and then apply Lemma \ref{7-link-nucleation-free}.  
    By assumption, the graph $H$ obtained from $G_i$ by deleting both of its hinge edges is nucleation-free.  
    Hence, any nucleation $N$ of $G'_i$ must contain the vertex $v$ added by the Henneberg-II extension performed on $G_i$.  
    Also, $N$ clearly must contain at least three edges incident to $v$.  
    If it contains exactly $3$ such edges, then applying what Theorem \ref{thm:prev_ops} says about Henneberg-I extensions to the independent nucleation-free graph obtained from $N$ by deleting $v$ shows that $N$ is nucleation-free, which is a contradiction.  
    
    Otherwise, wlog we can assume that $N$ is minimally rigid.  
    But then, the graph $N'$ obtained by reversing the Henneberg-II extension on $N$ -- i.e., deleting $v$ and making one hinge nonedge into an edge -- is minimally rigid.  
    This implies making the remaining hinge nonedge into an edge yields a dependent graph, which is a subgraph of $G_i$.  
    However, since $G_i$ is independent, this is a contradiction, and so $G'_i$ is nucleation-free.  

    Finally, we show that each hinge $f$ of $R'_m$ is implied using the rank-sandwich technique.  
    Let $R''_m = R'_m \cup S(\mathcal{X})$ and note that the $2$-thin cover $\mathcal{X} = \{X_1,\dots,X_m\}$ of $R'_m$, where $X_i = V(G'_i)$, is independent and also a $2$-thin cover of $R''_m$.  
    Since each $G_i$ is independent and rigid, by assumption, and Henneberg-II extensions preserve these properties, deleting any one hinge edge from any link $R''_m[X_i]$ yields an independent graph $H_i$ in which the deleted hinge is implied.  
    Consequently, we get $rank(R''_m[X_i]) = |E(G'_i)| + 1$.  
    Hence, Lemma \ref{thm:2-thin-rank} gives that $rank(R''_m) \leq \sum_{X_i \in \mathcal{X}} |E(G'_i)|$.  
    Since the RHS of this inequality is the number of edges in $R'_m$, we get $rank(R''_m) = rank(R'_m)$.  
    Finally, this implies that $rank(R'_m \cup f) = rank(R'_m)$, and thus $f$ is implied in $R'_m$.  
\end{proof}

\section{Split-and-glue}
\label{sec:split-and-glue}

Here we prove Theorems \ref{thm:drsg-ind} - \ref{thm:starter} in the following subsections.

\subsection{Nucleation-free preserving \SG/}
\label{sec:dra-nf}

Here we prove Theorem \ref{thm:sg-nf}.  

\begin{proof}[Proof of Theorem \ref{thm:sg-nf}]
    Assume that $G$ and $H$ are independent and $G^s$ and $H$ are nucleation-free.  
    Since any graph with a separator of size at most two is flexible, any nucleation $N'$ of $G^s:H$ contains at least three vertices in $\{a_1,a_2,b_1,b_2\}$ and is not a subgraph of $G^s$ nor $H$.  
    Wlog, we can assume that $|E(N')| = 3|V(N')| - 6$.  
    Let $J$ be the maximal induced subgraph of $N'$ contained in $H$ and let $N$ be the subgraph of $G$ obtained from $N'$ by first adding the vertices $a$ and $b$, then, for any vertex $w$ in $V(N') \setminus V(J)$, replacing all edges $(a_1,w)$ and $(a_2,w)$ with $(a,w)$ and all edges $(b_1,w)$ and $(b_2,w)$ with $(b,w)$, and finally deleting all vertices in $V(J) \cup \{a_1,a_2,b_1,b_2\}$.  
    Note that $|E(N)| = |E(N')| - |E(J)|$.  
    Hence, since $H$ is independent, we have $|E(N)| \geq 3|V(N') - 6 - 3(|V(J)| - 2)$.  
    Next, $J$ must contain at least three vertices in $\{a_1,a_2,b_1,b_2\}$, or else $N'$ has a separator of size two and consequently is flexible.  
    Therefore, we get that $|V(N')| - |V(N)| \geq |V(J)| - 2$.  
    Thus, we have $|E(N)| \geq 3|V(N') - 6 - 3(|V(N')| - |V(N)|) = 3|V(N)| - 6$, but then $G$ is dependent contrary to our assumption.  
    This completes the proof.  
\end{proof}



\subsection{Independence preserving double-butterfly \SG/}
\label{sec:dra-ind}

Here we prove Theorem \ref{thm:drsg-ind}.   
 Recall that each coordinate $s_{uv}$ of an equilibrium self-stress vector $s$ of a framework $G(p)$  corresponds to an edge $(u,v)$ of $G$, and an equilibrium self-stress  vector belongs to the left-null-space of the rigidity matrix $R(G(p))$ and satisfies the following \emph{stress balance equation} for each vertex $u$ of $G$:
$$\sum_{u:(u,v)\in E} s_{uv} (p_u - p_v) = 0.$$  
 Clearly, the rows of $R(G(p))$ are linearly independent if and only if $G(p)$ has no non-zero equilibrium self-stress.  
\begin{figure}[htb]
    \centering
    \begin{subfigure}[t]{0.49\linewidth}
        \centering
        \includegraphics[width=0.7\textwidth]{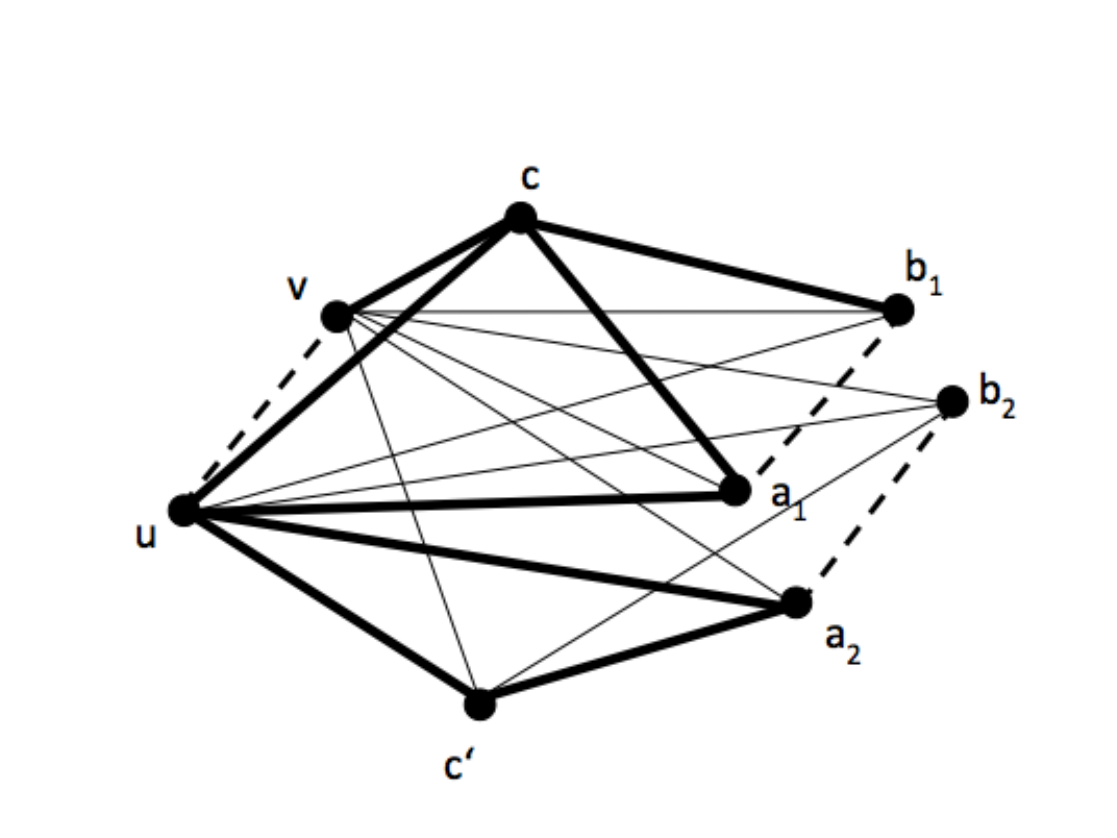}
    \end{subfigure}
    \begin{subfigure}[t]{0.49\linewidth}
        \centering
        \includegraphics[width=0.7\textwidth]{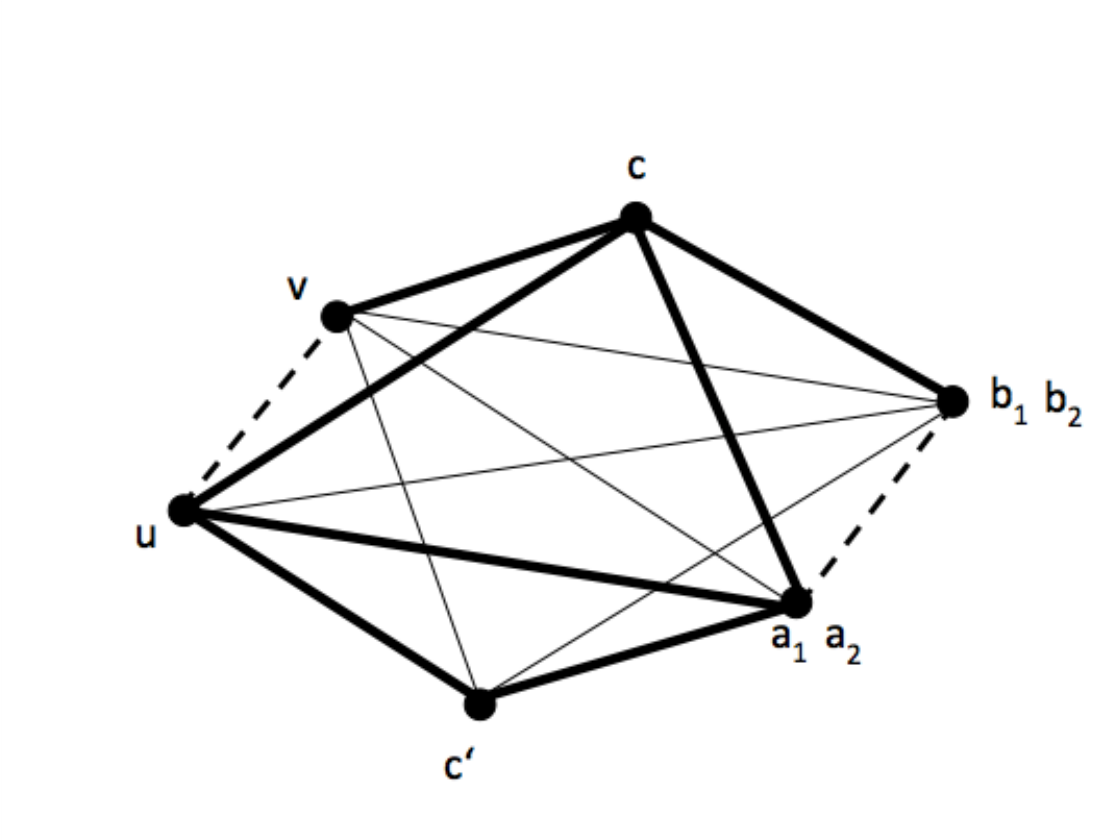}
    \end{subfigure}
    \begin{subfigure}[t]{0.49\linewidth}
        \centering
        \includegraphics[width=0.7\textwidth]{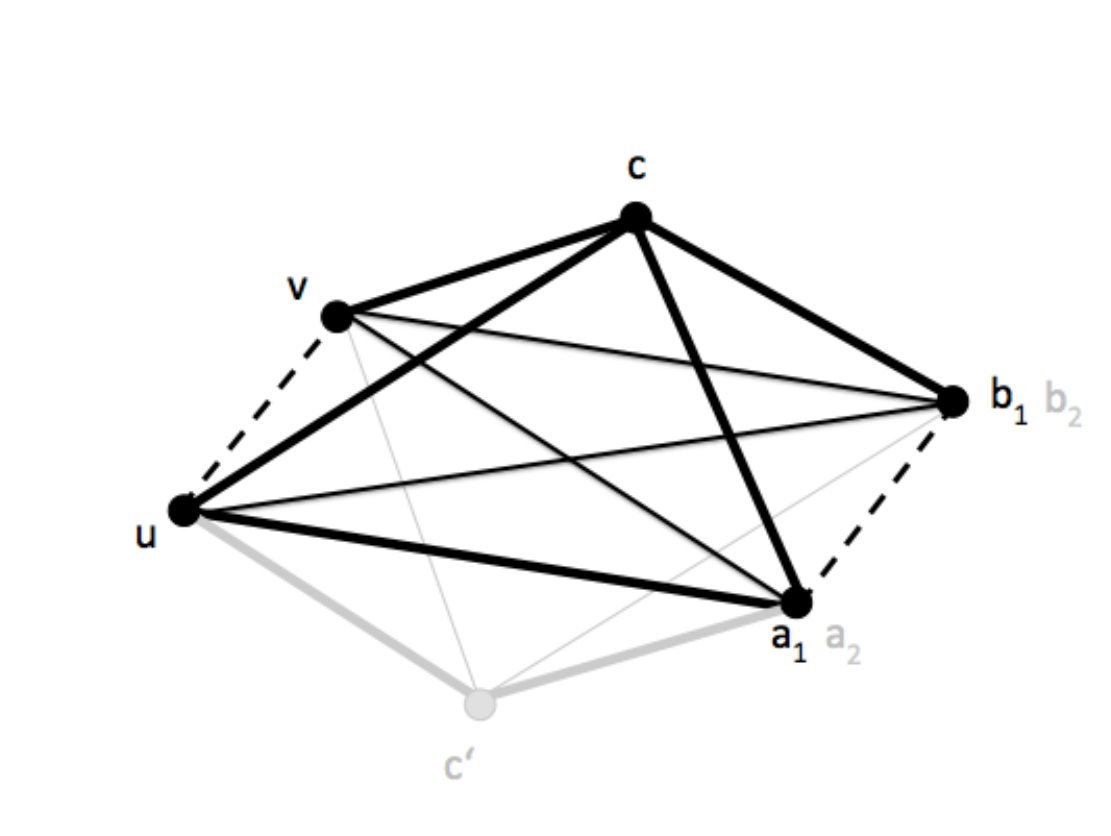}
    \end{subfigure}
    \begin{subfigure}[t]{0.49\linewidth}
        \centering
        \includegraphics[width=0.7\textwidth]{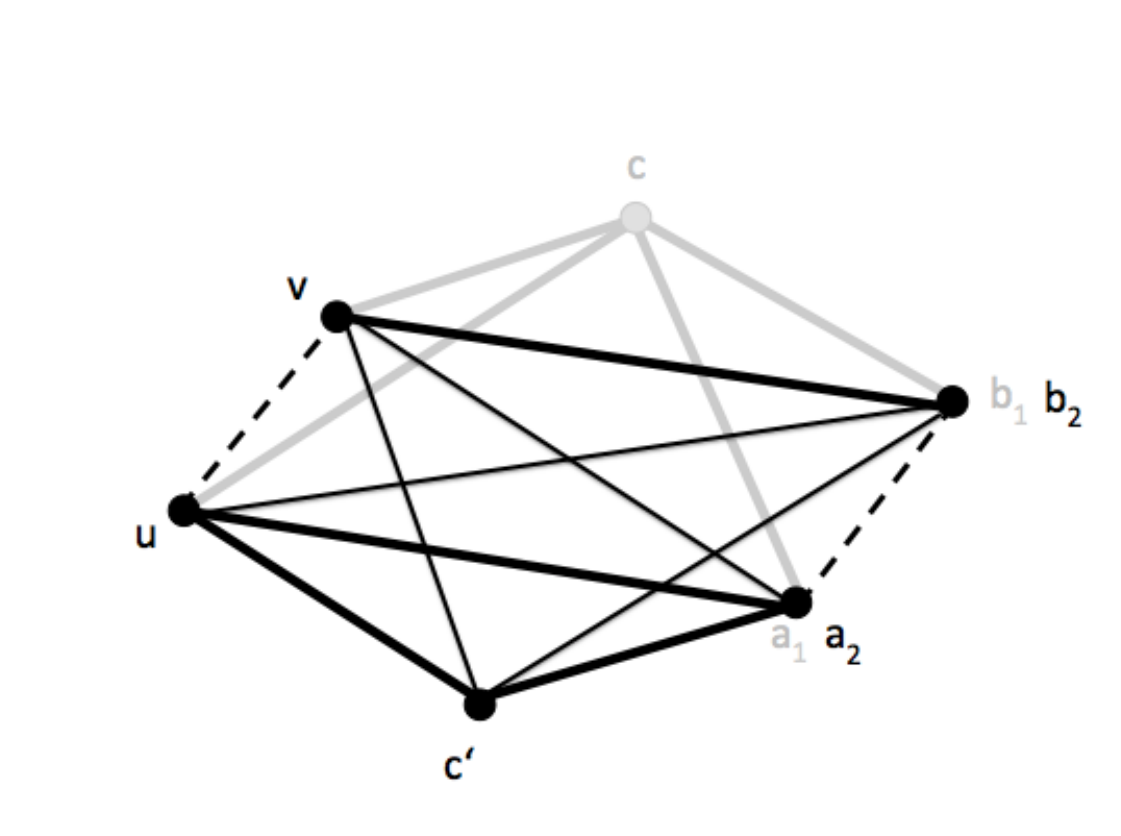}
    \end{subfigure}
    \caption{Top-left: a double-butterfly used as an ear in a double-butterfly \SG/.  
    Top-right: the placement of the double-butterfly in a hinged double-butterfly framework.  
    Bottom: the butterflies in the double-butterfly are highlighted.  
    See the proof sketch of Theorem \ref{thm:drsg-ind} in Section \ref{sec:split-and-glue}.}
    \label{fig:ringStress}
\end{figure}

\smallskip
We first provide the proof structure of Theorem  \ref{thm:drsg-ind} before the formal proof. 
Let a double-butterfly \SG/ whose base graph $G$ is independent split the nonedge $(a,b)$ into $(a_1,b_1)$ and $(a_2,b_2)$ and yield the graph $G'$.  
We start by taking any generic framework $G(q)$, which has no non-zero self-stress by the above fact, and using it to construct a particular (not necessarily generic) framework $G'(p)$, called a \emph{hinged double-butterfly framework}, as follows.  
Refer to Figure \ref{fig:ringStress}.  
For any vertex $x$ contained in both $G$ and $G'$, set $p_x = q_x$.  
Also, set $p_{a_1} = p_{a_2} = q_a$ and $p_{b_1} = p_{b_2} = q_b$.  
Letting the remaining vertices of the double-butterfly be labeled as in Figure \ref{fig:twoRoofs}, choose $p_u$ and $p_v$ so they form a square in a plane with $p_{a_1}$ and $p_{b_1}$ such that the line segments between $p_{a_1}$ and $p_v$ and between $p_{b_1}$ and $p_u$ intersect.  
Lastly, choose $p_c$ and $p_{c'}$ to be distinct and such that the line containing these points is perpendicular to the square and passes through its center.  

Next, using a series of lemmas, we show that $G'(p)$ has a non-zero self-stress only if $G(q)$ does.  
Since $G(q)$ has no non-zero self-stress, we get that the rigidity matrix of $G'(p)$ has linearly independent rows.  
Hence, this rigidity matrix has maximum rank over all rigidity matrices of frameworks of $G'$, and so there exists a generic framework sufficiently close to $G'(p)$ whose rigidity matrix has linearly independent rows.  
This implies that $G'$ is independent, which completes the proof.  
\qed

\smallskip
The relationship between the non-zero self-stresses of $G'(p)$ and $G(q)$ is established using Lemmas \ref{lem:dra-c} and \ref{lem:dra-a12}, below.  

\begin{lemma}[Stress values around $c$]
    \label{lem:dra-c}
    Consider a double-butterfly \SG/ yielding the graph $G'$, and let $G'(p)$ be a hinged double-butterfly framework.  
    Then, for any self-stress $s$ of $G'(p)$, we have $s_{ca_1} = s_{cv}$, $s_{cb_1} = s_{cu}$, and $s_{ca_1} = - s_{cb_1}$.  
\end{lemma}

\begin{proof}
    Wlog, assume that the square whose vertices are $q_{a_1}$, $q_{b_1}$, $q_u$, and $q_v$ lies in the $xy$-plane.  
    Also, let $xy(t)$ be the projection of any vector $t$ onto the $xy$-plane.  
    Since $q_c$ lies on the line perpendicular to the $xy$-plane and passing through the center of this square, we have that $xy(p_c - p_{a_1}) = - xy(p_c - p_v)$, $xy(p_c - p_{b_1}) = - xy(p_c - p_u)$, and $xy(p_c - p_{a_1})$ and $xy(p_c - p_{b_1})$ are perpendicular.  
    Also, since $s$ is a self-stress, we have 
    $$s_{ca_1}(p_c - p_{a_1}) + s_{cv}(p_c - p_v) + s_{cb_1}(p_c - p_{b_1}) + s_{cu}(p_c - p_u) = 0.$$  
    Combining these facts shows that $s_{ca_1} = s_{cv}$ and $s_{cb_1} = s_{cu}$, and hence the observation that $(p_c - p_{a_1})$, $(p_c - p_v)$, $(p_c - p_{b_1})$, and $(p_c - p_u)$ all have the same projection onto the $z$-axis shows that $s_{ca_1} = - s_{cb_1}$, which completes the proof.  
\end{proof}

\begin{lemma}[Stress values around $a_1$ and $a_2$]
    \label{lem:dra-a12}
    Consider a double-butterfly \SG/ yielding the graph $G'$, let $H$ be the double-butterfly, and let $G'(p)$ be a hinged double-butterfly framework.  
    Then, for any self-stress $s$ of $G'(p)$, the following equation is satisfied
    $$\sum_{(x,y)\in E(H):x\in\{a_1,a_2\}} s_{xy} (p_x - p_y) = 0.$$
\end{lemma}

\begin{proof}
    It suffices to show that the following equations are satisfied
    \begin{align}
        &s_{a_1c}(p_{a_1}-p_c) + s_{a_1v}(p_{a_1}-p_v) + s_{a_2c'}(p_{a_2}-p_{c'}) + s_{a_2v}(p_{a_2}-p_v) = 0,\label{eqn:1}\\
        &s_{a_1u}(p_{a_1}-p_u) + s_{a_2u}(p_{a_2}-p_u) = 0.\label{eqn:2}
    \end{align}
    To show that \ref{eqn:1} is satisfied, we first demonstrate that \begin{align}
        &s_{vc}(p_v-p_c) + s_{a_1v}(p_v-p_{a_1}) + s_{vc'}(p_v-p_{c'}) + s_{a_2v}(p_v-p_{a_2}) = 0,\label{eqn:3}\\
        &s_{vb_1}(p_v-p_{b_1}) + s_{vb_2}(p_v-p_{b_2}) = 0\label{eqn:4}
    \end{align}
    are satisfied.  
    Since $s$ is a self-stress, we have that
    $$\sum_{y:(v,y)\in E(G')} s_{vy} (p_v - p_y) = 0.$$  
    Note that some plane $P$ contains all four vectors in \ref{eqn:3}.  
    Since $(p_v-p_{b_1})$ and $(p_v-p_{b_2})$ are equal and have the same non-zero projection onto any plane perpendicular to $P$, the above facts show that \ref{eqn:4} is satisfied, and consequently \ref{eqn:3} is satisfied.  
    Combining this with Lemma \ref{lem:dra-c} and the reflectional symmetry of the double-butterfly framework in $G'(p)$ shows that \ref{eqn:1} is satisfied.  

    A similar argument that examines stresses on the edges incident on $u$ shows that \ref{eqn:2} is satisfied.  
    This competes the proof.  
\end{proof}

We now prove Theorem \ref{thm:drsg-ind}.  

\begin{proof}[Proof of Theorem \ref{thm:drsg-ind}]
    Consider any double-butterfly \SG/ whose base graph $G$ is independent and that yields the graph $G'$.  
    As discussed in the above proof sketch, it suffices to show that some hinged double-butterfly framework $G'(p)$ has no non-zero self-stress.  
    We will show that if such a self-stress exists, then the generic framework $G(q)$ from which $G'(p)$ is constructed has a non-zero self-stress.  
    Then, combining the contrapositive of this statement with the independence of $G$ completes the proof.  

    Let $s'$ be a non-zero self-stress of $G'(p)$.  
    Also, let $s$ be the vector whose coordinates $s_{xy}$ correspond to some edge $(x,y)$ of $G$ such that $s_{xy} = s'_{xy}$ if $(x,y)$ is not incident to $a$ or $b$, $s_{xy} = s'_{a_iy}$ if $x = a$ and $(a_i,y)$ is an edge of $G'$, and $s_{xy} = s'_{b_iy}$ if $x = b$ and $(b_i,y)$ is an edge of $G'$.  
    Using Lemma \ref{lem:dra-a12}, we see that $s$ is a self-stress of $G(q)$.  
    It remains to show that $s$ is non-zero.  
    Let $H$ be the double-butterfly ear of the \SG/ and let $t$ be the projection of $s'$ onto the coordinates corresponding to the edges of $H$.  
    If $t$ is the zero-vector, then the fact that $s'$ is non-zero shows that $s$ is non-zero.  
    Otherwise, we consider each case for the non-zero coordinates of $t$.  
    Wlog, assume that the square whose vertices are $p_{a_1}$, $p_{b_1}$, $p_u$, and $p_v$ lies in the $xy$-plane with $p_u$ at the origin, $p_{a_1}$ on the positive $x$-axis, and $p_v$ on the positive $y$-axis.  
    
    Assume that $t_{a_1c}$ is non-zero.  
    Then, the vector $w = s'_{a_1c}(p_{a_1}-p_c) + s'_{a_1u}(p_{a_1}-p_u) + s'_{a_1v}(p_{a_1}-p_v)$ has a non-zero $z$-coordinate.  
    Since $s'$ is a self-stress, this implies that $s$ has a non-zero coordinate corresponding to some edge incident to $a$.  
    An identical argument applies if $t_{a_2c'}$, $t_{b_1c}$, or $t_{b_2c'}$ is non-zero.  
    Also, if $t_{uc}$, $t_{uc'}$, $t_{vc}$, or $t_{vc'}$ is non-zero, then Lemma \ref{lem:dra-c} shows that at least one of $t_{a_1c}$, $t_{a_2c'}$, $t_{b_1c}$, or $t_{b_2c'}$ is non-zero, and so the above argument applies.  

    Next, assume that none of the above situations occur.  
    Also, assume that $t_{a_1v}$ is non-zero.  
    Since $t_{a_1c}$ is zero, $w$ has a non-zero $y$-coordinate.  
    Since $s'$ is a self-stress, this implies that $s$ has a non-zero coordinate corresponding to some edge incident to $a$.  
    An identical argument applies if $t_{a_2v}$, $t_{b_1u}$, or $t_{b_2u}$ is non-zero.  
    Lastly, a similar argument applies to the case where $t_{a_1u}$, $t_{a_2u}$, $t_{b_1v}$, or $t_{b_2v}$ is non-zero.  
    This completes the proof.  
\end{proof}

\subsection{Implied nonedge preserving safe \SG/}
\label{sec:dra-implied}

Here we prove Theorem \ref{thm:sg-implied}.  
First, we define a \emph{safe \SG/}, which is a \SG/ with additional conditions on all involved graphs and operations.  
We introduce these conditions in parts.

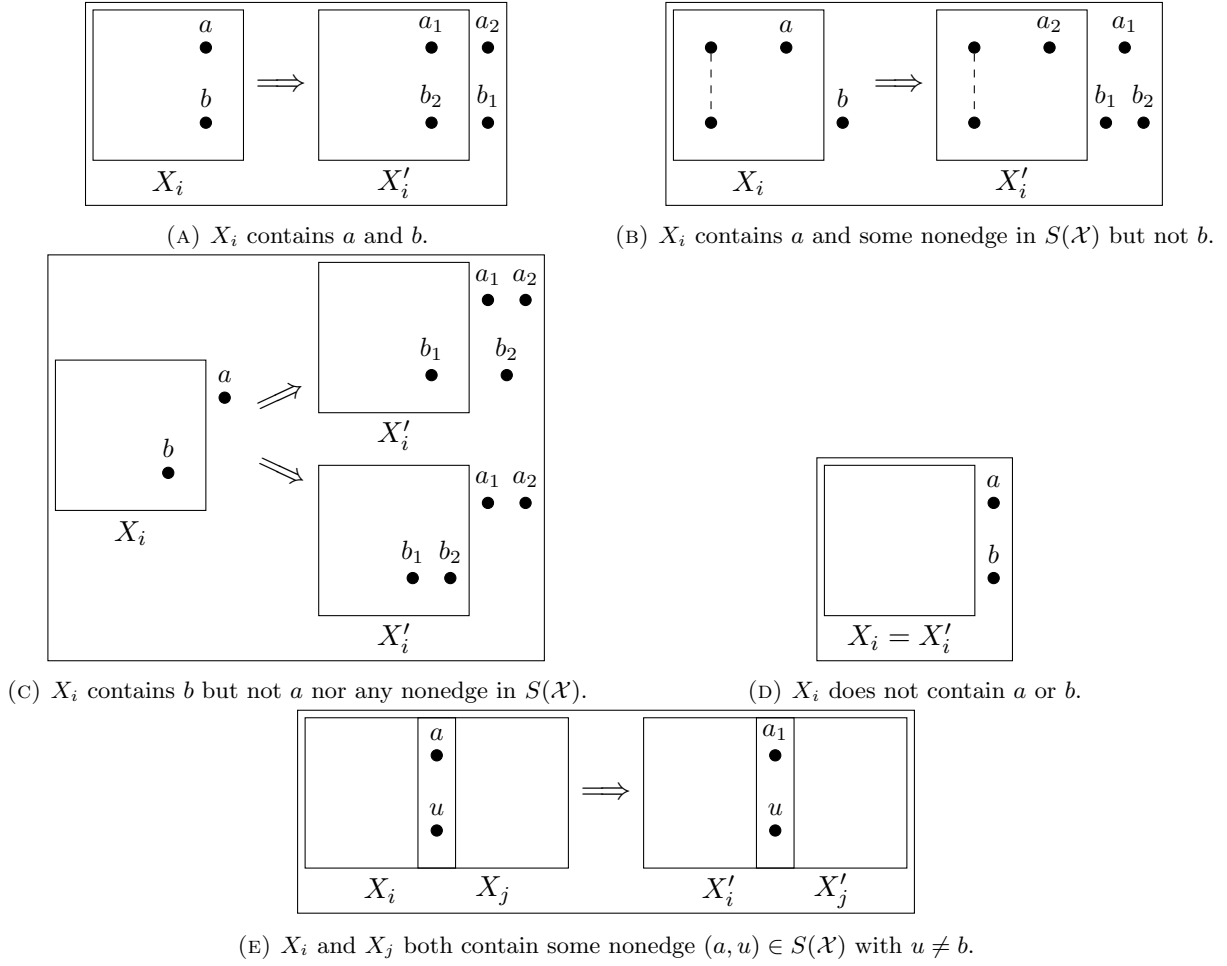
\begin{figure}[htb]
    \centering
    \begin{subfigure}[t]{0.49\linewidth}
        \centering
        \begin{tikzpicture}
            \draw (-0.1,-0.6) rectangle (5.5,2.1);
            
            \draw (0,0) rectangle (2,2);
            \node at (1,-0.3) {\large$X_i$};
            \node[draw,circle,fill,inner sep=-1.5pt, label=$a$] (a) at (1.5,1.5) {};
            \node[draw,circle,fill,inner sep=-1.5pt, label=$b$] (b) at (1.5,0.5) {};
    
            \node at (2.5,1) {\Large$\implies$};
    
            \draw (3,0) rectangle (5,2);
            \node at (4,-0.3) {\large$X'_i$};
            \node[draw,circle,fill,inner sep=-1.5pt, label=$a_1$] (a) at (4.5,1.5) {};
            \node[draw,circle,fill,inner sep=-1.5pt, label=$a_2$] (a) at (5.25,1.5) {};
            \node[draw,circle,fill,inner sep=-1.5pt, label=$b_2$] (b) at (4.5,0.5) {};
            \node[draw,circle,fill,inner sep=-1.5pt, label=$b_1$] (b) at (5.25,0.5) {};
        \end{tikzpicture}
        \caption{$X_i$ contains $a$ and $b$.}
    \end{subfigure}
    \begin{subfigure}[t]{0.49\linewidth}
        \centering
        \begin{tikzpicture}
            \draw (-0.1,-0.6) rectangle (6.5,2.1);
            
            \draw (0,0) rectangle (2,2);
            \node at (1,-0.3) {\large$X_i$};
            \node[draw,circle,fill,inner sep=-1.5pt] (u) at (0.5,1.5) {};
            \node[draw,circle,fill,inner sep=-1.5pt] (v) at (0.5,0.5) {};
            \draw[dashed] (u) -- (v);
            \node[draw,circle,fill,inner sep=-1.5pt, label=$a$] (a) at (1.5,1.5) {};
            \node[draw,circle,fill,inner sep=-1.5pt, label=$b$] (b) at (2.25,0.5) {};
    
            \node at (3,1) {\Large$\implies$};
    
            \draw (3.5,0) rectangle (5.5,2);
            \node at (4.5,-0.3) {\large$X'_i$};
            \node[draw,circle,fill,inner sep=-1.5pt] (u) at (4,1.5) {};
            \node[draw,circle,fill,inner sep=-1.5pt] (v) at (4,0.5) {};
            \draw[dashed] (u) -- (v);
            \node[draw,circle,fill,inner sep=-1.5pt, label=$a_2$] (a) at (5,1.5) {};
            \node[draw,circle,fill,inner sep=-1.5pt, label=$a_1$] (a) at (6,1.5) {};
            \node[draw,circle,fill,inner sep=-1.5pt, label=$b_1$] (b) at (5.75,0.5) {};
            \node[draw,circle,fill,inner sep=-1.5pt, label=$b_2$] (b) at (6.25,0.5) {};
        \end{tikzpicture}
        \caption{$X_i$ contains $a$ and some nonedge in $S(\mathcal{X})$ but not $b$.}
    \end{subfigure}
    \begin{subfigure}[t]{0.49\linewidth}
        \centering
        \begin{tikzpicture}
            \draw (-0.1,-3.2) rectangle (6.5,2.2);
            
            \draw (0,-1.2) rectangle (2,0.8);
            \node at (1,-1.5) {\large$X_i$};
            \node[draw,circle,fill,inner sep=-1.5pt, label=$a$] (a) at (2.25,0.3) {};
            \node[draw,circle,fill,inner sep=-1.5pt, label=$b$] (b) at (1.5,-0.7) {};
    
            \node[rotate=25] at (3,0.3) {\Large$\implies$};
            \node[rotate=-25] at (3,-0.7) {\Large$\implies$};
    
            \draw (3.5,0.1) rectangle (5.5,2.1);
            \node at (4.5,-0.2) {\large$X'_i$};
            \node[draw,circle,fill,inner sep=-1.5pt, label=$b_1$] (a) at (5,0.6) {};
            \node[draw,circle,fill,inner sep=-1.5pt, label=$b_2$] (a) at (6,0.6) {};
            \node[draw,circle,fill,inner sep=-1.5pt, label=$a_1$] (b) at (5.75,1.6) {};
            \node[draw,circle,fill,inner sep=-1.5pt, label=$a_2$] (b) at (6.25,1.6) {};

            \draw (3.5,-2.6) rectangle (5.5,-0.6);
            \node at (4.5,-2.9) {\large$X'_i$};
            \node[draw,circle,fill,inner sep=-1.5pt, label=$b_1$] (a) at (4.75,-2.1) {};
            \node[draw,circle,fill,inner sep=-1.5pt, label=$b_2$] (a) at (5.25,-2.1) {};
            \node[draw,circle,fill,inner sep=-1.5pt, label=$a_1$] (b) at (5.75,-1.1) {};
            \node[draw,circle,fill,inner sep=-1.5pt, label=$a_2$] (b) at (6.25,-1.1) {};
        \end{tikzpicture}
        \caption{$X_i$ contains $b$ but not $a$ nor any nonedge in $S(\mathcal{X})$.}
    \end{subfigure}
    \begin{subfigure}[t]{0.49\linewidth}
        \centering
        \begin{tikzpicture}
            \draw (-0.1,-0.6) rectangle (2.5,2.1);
            
            \draw (0,0) rectangle (2,2);
            \node at (1,-0.3) {\large$X_i = X'_i$};
            \node[draw,circle,fill,inner sep=-1.5pt, label=$a$] (a) at (2.25,1.5) {};
            \node[draw,circle,fill,inner sep=-1.5pt, label=$b$] (b) at (2.25,0.5) {};
        \end{tikzpicture}
        \caption{$X_i$ does not contain $a$ or $b$.}
    \end{subfigure}
    \begin{subfigure}[t]{\linewidth}
        \centering
        \begin{tikzpicture}
            \draw (-0.1,-0.6) rectangle (8.1,2.1);
            
            \draw (0,0) rectangle (2,2);
            \node at (1,-0.3) {\large$X_i$};
            \node[draw,circle,fill,inner sep=-1.5pt, label=$a$] (a) at (1.75,1.5) {};
            \node[draw,circle,fill,inner sep=-1.5pt, label=$u$] (u) at (1.75,0.5) {};
             \draw (1.5,0) rectangle (3.5,2);
            \node at (2.5,-0.3) {\large$X_j$};

            \node at (4,1) {\Large$\implies$};

            \draw (4.5,0) rectangle (6.5,2);
            \node at (5.5,-0.3) {\large$X'_i$};
            \node[draw,circle,fill,inner sep=-1.5pt, label=$a_1$] (a) at (6.25,1.5) {};
            \node[draw,circle,fill,inner sep=-1.5pt, label=$u$] (u) at (6.25,0.5) {};
             \draw (6,0) rectangle (8,2);
            \node at (7,-0.3) {\large$X'_j$};
        \end{tikzpicture}
        \caption{$X_i$ and $X_j$ both contain some nonedge $(a,u) \in S(\mathcal{X})$ with $u \neq b$.  }
    \end{subfigure}
    \caption{All cases for $X'_i$ in a safe nonedge-split on a graph $G$ that splits $(a,b)$ into $(a_1,b_1)$ and $(a_2,b_2)$, up to interchanging $a$, $b$, and indices.  
    The safe split cover $\mathcal{X}^s$ contains $X'_i$ in all cases except the bottom-right case of (c), where each pair of distinct vertices in $X'_i$ that is an edge of $G$ is a cluster of $\mathcal{X}^s$.  }
    \label{fig:safesplit}
\end{figure}

A \emph{safe nonedge-split} on a pair $(G,\mathcal{X})$, where $\mathcal{X} = \{X_1,\dots,X_m\}$ is a $2$-thin cover of a graph $G$, is a nonedge-split on $G$ yielding $G^s$ that splits a nonedge $(a,b)$ in the shared set $S(\mathcal{X})$ into $(a_1,b_1)$ and $(a_2,b_2)$ such that the following two conditions are satisfied.  
Let $X'_i = (X_i \setminus \{a,b\}) \cup W_i$, where $W_i$ is the subset of all vertices in $\{a_1,a_2,b_1,b_2\}$ adjacent in $G^s$ to some vertex in $X_i$.  
(i) Neither $\{a_1,a_2\}$ nor $\{b_1,b_2\}$ is a subset of $X'_i$ whenever $X_i$ contains some nonedge in $S(\mathcal{X})$ and (ii) $X'_i \cap \{a_1,a_2\} = X'_j \cap \{a_1,a_2\}$ if $X_i \cap X_j = \{a,u\}$ where $u \neq b$ and $(a,u)$ is a nonedge in $S(\mathcal{X})$, and the same holds after interchanging $a$ and $b$.  
This operation induces a $2$-thin cover $\mathcal{X}^s$ of $G^s$, called the \emph{safe split cover}, whose sets of size greater than two are exactly those $X'_i$ that contain neither $\{a_1,a_2\}$ nor $\{b_1,b_2\}$ as a subset and whose sets of size two are the endpoints of an edge of $G^s$.  
See Figure \ref{fig:safesplit} and Figure \ref{fig:safesg}(a)-(c).  

\begin{figure}[htb]
    \centering
    \includegraphics[width=0.3\linewidth]{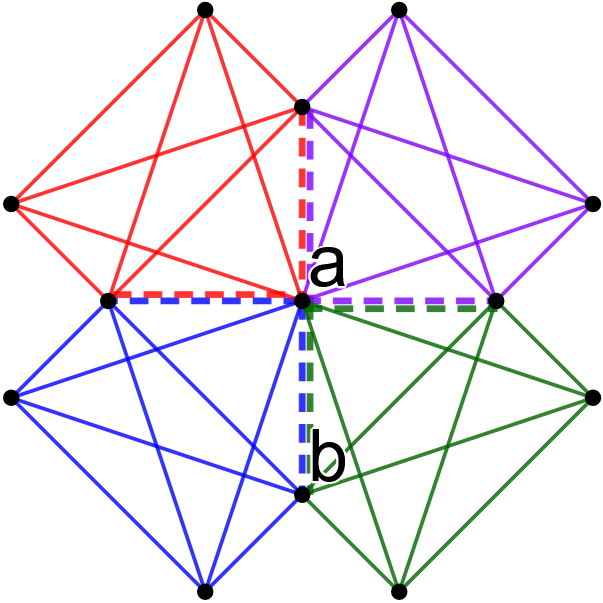}
    \caption{A graph $G$ and $2$-thin cover $\mathcal{X}$, whose clusters are colored and whose shared nonedges are depicted as parallel dashed line-segments of different colors, such that all safe nonedge-splits on $(G,\mathcal{X})$ that split $(a,b)$ are trivial.  
    It is easy to check that $G$ is a safe base graph.  
    See the definitions in Section \ref{sec:dra-implied}.  
    }
    \label{fig:nontrivialSG}
\end{figure}

Note that a nonedge-split that assigns all edges to exactly two vertices in $\{a_1,a_2,b_1,b_2\}$ is always safe, and we refer to this as a \emph{trivial} nonedge-split.  
However, it may not be possible to perform a non-trivial safe nonedge-split on a given pair $(G,\mathcal{X})$, i.e., every safe nonedge-split is trivial.  
For example, if $(G,\mathcal{X})$ is as in Figure \ref{fig:nontrivialSG} and some cluster of the safe split cover $\mathcal{X}^s$ contains $a_1$ and $b_1$ but not $a_2$ or $b_2$, then all clusters of $\mathcal{X}^s$ contain $a_1$ and $b_1$ but not $a_2$ or $b_2$
Thus, every safe nonedge-split is trivial.  
It is easy to check that $G$ is a \emph{safe base graph}, defined next, and so this obstacle to a non-trivial safe nonedge-split exists even when the conditions in this definition are satisfied.  

A \emph{safe base graph} $G$ is an independent graph that has some independent $2$-thin cover $\mathcal{X}$ such that $rank(G) = IE(G,\mathcal{X})$ and the shared set $S(\mathcal{X})$ contains a nonedge $(a,b)$ that can be split by a safe nonedge-split on $(G,\mathcal{X})$.  
Let $G^s$ be the split graph and $\mathcal{X}^s$ be the safe split cover resulting from this safe nonedge-split.  
A pair of distinct vertices in $\{a_1,a_2,b_1,b_2\}$ is a \emph{key gluing pair} if it is contained in some cluster of $\mathcal{X}^s$.  
Note that the definition of $\mathcal{X}^s$ implies that at least one key gluing pair exists.  
The \emph{gluing vertices} of $G^s$ are a subset of $\{a_1,a_2,b_1,b_2\}$ that contains the endpoints of all key gluing pairs.  
See Figure \ref{fig:safesg}(a)-(c).  

\begin{remark}
    The pair $(a,b)$ is an implied nonedge in $G$ because $IE(G,\mathcal{X})$ is the same whether the pair $(a,b)$ is an edge or nonedge, is equal to the rank of $G$, and is an upper bound on the rank of $G \cup {(a,b)}$.
\end{remark}





A \emph{safe ear} $H$ is a graph that has some independent $2$-thin cover $\mathcal{X}^H$, called a \emph{safe ear cover}, and some set of \emph{gluing vertices} satisfying the following properties.  
Let $IE'(H,\mathcal{X}^H)$ be obtained from $IE(H,\mathcal{X}^H)$ by adding all \emph{key gluing pairs}, i.e., pairs of distinct gluing vertices that are contained in some cluster of $\mathcal{X}^H$, to $S(\mathcal{X}^H)$ in the summations.  
(i) Each cluster of $\mathcal{X}^H$ contains at most two gluing vertices, (ii) $S(\mathcal{X}^H) \cup F$ is independent where $F$ is the set of all pairs of distinct gluing vertices, and (iii) $rank(H) = IE'(H,\mathcal{X}^H) - 1$.  
See Figure \ref{fig:safesg}(d).  

\begin{figure}[ht!]
    \centering
    \begin{subfigure}[t]{0.49\linewidth}
        \centering
        \includegraphics[width=0.5\linewidth]{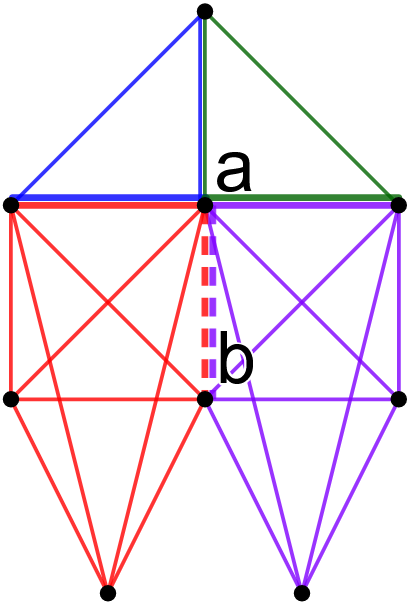}
        \caption{}
    \end{subfigure}
    \begin{subfigure}[t]{0.49\linewidth}
        \centering
        \includegraphics[width=0.8\linewidth]{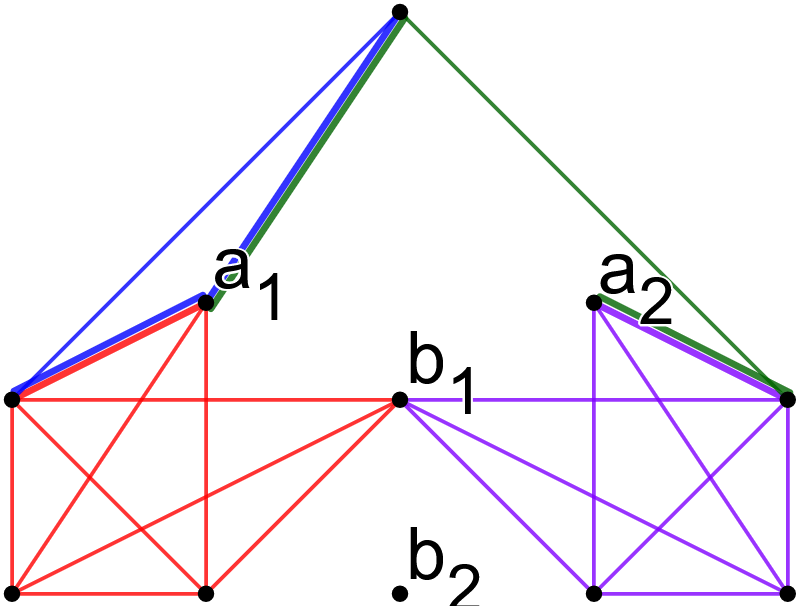}
        \caption{}
    \end{subfigure}
    \begin{subfigure}[t]{0.49\linewidth}
        \centering
        \includegraphics[width=0.8\linewidth]{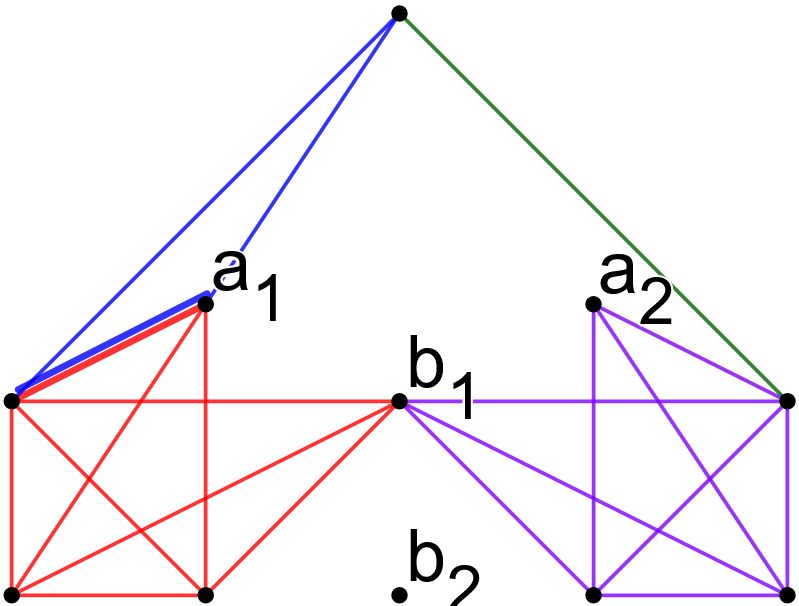}
        \caption{}
    \end{subfigure}
    \begin{subfigure}[t]{0.49\linewidth}
        \centering
        \includegraphics[width=0.5\linewidth]{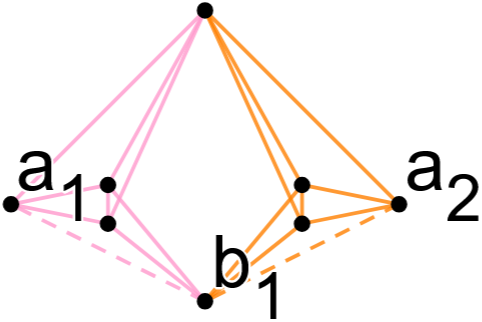}
        \caption{}
    \end{subfigure}
    \begin{subfigure}[t]{0.49\linewidth}
        \centering
        \includegraphics[width=0.8\linewidth]{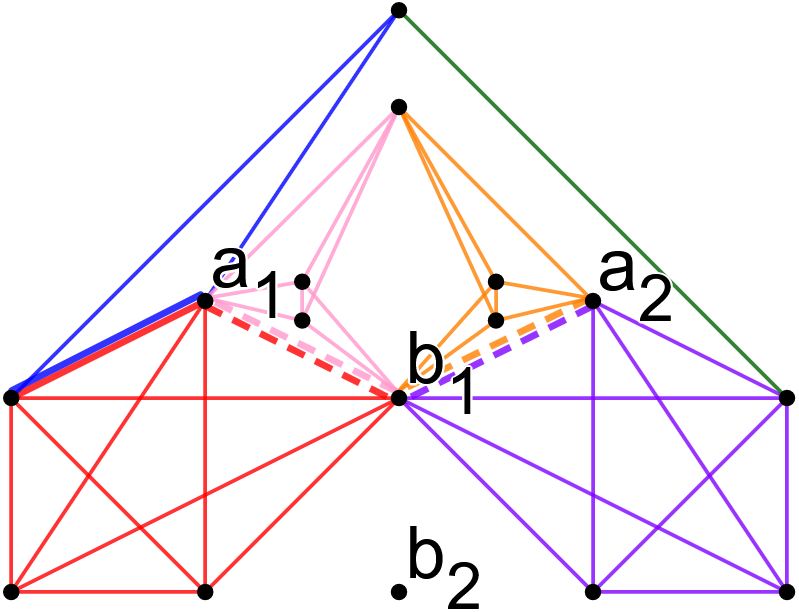}
        \caption{}
    \end{subfigure}
    \caption{Shared (non)edges are depicted as parallel (dashed) line-segments of different colors.  
    (a) A safe base graph with a $2$-thin $\mathcal{X}$ cover consisting of four clusters.  
    (b) The split graph and sets $X'_i$ resulting from a non-trivial safe nonedge-split on $(G,\mathcal{X})$ that splits $(a,b)$.  
    (c) The safe split cover of the split graph.  
    The green cluster in (b) contains both $a_1$ and $a_2$, and so the endpoints of each unshared edge it contains is made into a cluster in (c), which is just one edge in this case.  
    The key gluing pairs are $(a_1,b_1)$ and $(a_2,b_1)$.  
    (d) A safe ear and 2-thin cover with key gluing pairs $(a_1,b_1)$ and $(a_2,b_1)$.  
    (e) The graph and 2-thin cover resulting from a safe \SG/ on the graphs in (c) and (d).  
    }
    \label{fig:safesg}
\end{figure}

A \emph{safe \SG/} on a safe base graph $G$ and safe ear $H$ consists of a safe nonedge-split on $(G,\mathcal{X})$, yielding $G^s$, followed by a glue on $G^s$ and $H$ using their gluing vertices that bijectively identifies the key gluing pairs of $G$ with the key gluing pairs of $H$.  
Let $\mathcal{X}^s$ and $\mathcal{X}^H$ be the safe split and safe ear covers.  
This operation induces a $2$-thin cover $\mathcal{X}'$ of the resulting graph that consists of all maximal clusters of $\mathcal{X}^s \cup \mathcal{X}^H$ (specifically, clusters of size 2 in $\mathcal{X}^H$ are dropped if after gluing they belong to a cluster in $\mathcal{X}^s$).  
See Figure \ref{fig:safesg}.

\begin{lemma}[Safe \SG/ covers are independent]
    \label{lem:sg-cover-ind}
    For any safe \SG/ yielding the graph $G^s:H$ and its cover $\mathcal{X}'$, $\mathcal{X}^s$ and $\mathcal{X}'$ are independent $2$-thin covers of $G^s$ and $G^s:H$, respectively.  
\end{lemma}

\begin{proof}
    By definition, $\mathcal{X}^s$ is a $2$-thin cover of $G^s$.  
    Also, since at most two gluing vertices of the safe ear are contained in any cluster of the safe ear cover, we get that $\mathcal{X}'$ is a $2$-thin cover of $G^s:H$.  

    For independence, consider the graph $I$ obtained from $(V(G),S(\mathcal{X}))$ via a $1$-vertex split that splits $a$ into $a_1$ and $a_2$ followed by a $2$-vertex split that splits $b$ into $b_1$ and $b_2$.  
    We get that $\mathcal{X}^s$ is independent since $(V(G^s), S(\mathcal{X}^s))$ is a subgraph of $I$, $\mathcal{X}$ is independent, and these vertex splits preserve independence.  
    Lastly, let $K$ and $K'$ be cliques on the the gluing vertices of $G^s$ and $H$, respectively.  
    Observe that $(V(G^s:H),S(\mathcal{X}'))$ is a subgraph of the $k$-sum of $I$ and $(V(H),S(\mathcal{X}^H) \cup E(K'))$ such that the base cliques are $K$ and $K'$.  
    Along with the facts that $(V(H),S(\mathcal{X}^H) \cup E(K'))$ is independent, by the definition of a safe ear, and $k$-sums preserve independence, we get that $\mathcal{X}'$ is independent.  
\end{proof}



\begin{figure}
    \centering
    \begin{tikzpicture}
        \node[draw] (rank) at (0,0) {$rank(G^s:H)$};

        \node[draw] (edge) at (0,-1.5) {$|E(G^s:H)|$};

        \node[draw] (baseedge) at (-3,-3) {$|E(G)|$};
        \node (pedge) at (0,-3) {\huge$+$};
        \node[draw] (earedge) at (3,-3) {$|E(H)|$};

        \node[draw] (baserank) at (-3,-4.5) {$rank(G)$};
        \node (prank) at (0,-4.5) {\huge$+$};
        \node[draw] (earrank) at (3,-4.5) {$rank(H)$};

        \node[draw] (baseie) at (-3,-6) {$IE(G,\mathcal{X})$};
        
        \node[draw] (splitie) at (-3,-7.5) {$IE'(G^s,\mathcal{X}^s) - |F| + 1$};
        \node (pie) at (0,-7.5) {\huge$+$};
        \node[draw] (earie) at (3,-7.5) {$IE'(H,\mathcal{X}^H) - 1$};
        
        \node[draw] (ie) at (0,-9) {$IE(G^s:H,\mathcal{X}')$};

        \draw[draw=none] (rank) -- node[sloped] {\huge$=$} (edge);
        \draw[draw=none] (rank) -- node[right,xshift=4pt] {independence} (edge);
        
        \draw[draw=none] (edge) -- node[sloped] {\huge$=$} (pedge);
        \draw[draw=none] (edge) -- node[right,xshift=4pt] {construction} (pedge);

        \draw[draw=none] (baseedge) -- node[sloped] {\huge$=$} (baserank);
        \draw[draw=none] (baseedge) -- node[right,xshift=4pt] {independence} (baserank);

        \draw[draw=none] (earedge) -- node[sloped] {\huge$=$} (earrank);
        \draw[draw=none] (earedge) -- node[right,xshift=4pt] {independence} (earrank);

        \draw[draw=none] (earrank) -- node[sloped] {\huge$=$} (earie);
        \draw[draw=none] (earrank) -- node[right,xshift=4pt] {safe ear} (earie);
        
        \draw[draw=none] (baserank) -- node[sloped] {\huge$=$} (baseie);
        \draw[draw=none] (baserank) -- node[right,xshift=4pt] {safe base graph} (baseie);

        \draw[draw=none] (earrank) -- node[sloped] {\huge$=$} (earie);
        \draw[draw=none] (earrank) -- node[right,xshift=4pt] {safe ear} (earie);
        
        \draw[draw=none] (baseie) -- node[sloped] {\huge$=$} (splitie);
        \draw[draw=none] (baseie) -- node[right,xshift=4pt] {Lemma \ref{lem:edge-count} proof} (splitie);

        \draw (5,0.2) -- node[right,xshift=4pt] {Lemma \ref{lem:edge-count}} (5,-7.7);
        
        \draw[draw=none] (pie) -- node[sloped] {\huge$=$} (ie);
        \draw[draw=none] (pie) -- node[right,xshift=4pt] {Lemma \ref{lem:ie-count}} (ie);
    \end{tikzpicture}
    \caption{Proof outline for Theorem \ref{thm:sg-implied} in Section \ref{sec:dra-implied}.}
    \label{fig:implied-outline}
\end{figure}

The remainder of the proof of Theorem \ref{thm:sg-implied} is outlined in Figure \ref{fig:implied-outline}.  
For any safe \SG/ yielding the graph $G^s:H$ and its cover $\mathcal{X}'$, let $IE'(G^s,\mathcal{X}^s)$ be obtained from $IE(G^s,\mathcal{X}^s)$ by adding all key gluing pairs of $G^s$ to $S(\mathcal{X}^s)$ in the summations.  

\begin{lemma}[Rank of safe \SG/ graph]
    \label{lem:edge-count}
    Consider a safe \SG/ yielding the graph $G^s:H$ and its cover $\mathcal{X}'$, and let $F$ be the set of key gluing pairs of $G^s$.  
    If $G^s:H$ is independent, then
    $$rank(G^s:H) = IE'(G^s,\mathcal{X}^s) + IE'(H,\mathcal{X}^H) - |F|.$$
\end{lemma}

\begin{proof}
    Assume $G^{sg} = G^s:H$ is independent.  
    Then, $rank(G^{sg}) = |E(G^{sg})|$ and $H$ is independent, and hence $rank(H) = |E(H)|$.  
    Also, since $G$ is independent, by definition, we have $rank(G) = |E(G)|$.  
    Combining these facts with the observation that $|E(G^{sg})| = |E(G)| + |E(H)|$ and the rank conditions on $G$ and $H$ tells us that 
    \begin{align*}
        rank(G^{sg}) = IE(G,\mathcal{X}) + IE'(H,\mathcal{X}^H) - 1.
    \end{align*}
    Hence, it suffices to show that $IE(G,\mathcal{X}) = IE'(G^s,\mathcal{X}^s) - |F| + 1$.  
    Let the safe \SG/ split the nonedge $(a,b)$ into the nonedges $(a_1,b_1)$ and $(a_2,b_2)$.  
    Let $\mathcal{X}_{pres} \subseteq \mathcal{X}$ be the set of all clusters that can be obtained from some cluster in $\mathcal{X}^s$ of size at least three by replacing $a_1$ or $a_2$ with $a$ and replacing $b_1$ or $b_2$ with $b$, if this cluster contains any of these vertices (i.e., $\mathcal{X}_{pres}$ contains all non-trivial clusters ``preserved'' by the split).  
    Also, let $E(\mathcal{X}_{pres})$ be the set of edges of $G$ induced by the union of all sets in $\mathcal{X}_{pres}$ and let $S(\mathcal{X}_{pres}) \subseteq S(\mathcal{X})$ contain all pairs in the intersection of at least two sets in $\mathcal{X}_{pres}$.  
    First, we show that
    \begin{align}
        \begin{split}
            \label{eqn:5}
            IE(G,\mathcal{X}) &= 
            \sum_{X_i \in \mathcal{X}_{pres}} rank(G_{S(\mathcal{X})}[X_i]) 
            + \sum_{e \in E(G) \setminus E(\mathcal{X}_{pres})} 1\\ 
            &- \sum_{e \in S(\mathcal{X}_{pres})} (d(\mathcal{X}_{pres},e) - 1). 
        \end{split}
    \end{align}
    Let $\mathcal{X}_{triv} \subseteq \mathcal{X}$ be the set of all size two clusters (i.e. trivial clusters), $\mathcal{X}_{dest} = \mathcal{X} \setminus (\mathcal{X}_{pres} \cup \mathcal{X}_{triv})$ (i.e. non-trivial clusters ``destroyed'' by the split), and $\overline{S(\mathcal{X}_{pres})} = S(\mathcal{X}) \setminus S(\mathcal{X}_{pres})$.  
    Then, 
    \begin{align*}
        IE(G,\mathcal{X}) &= 
        \sum_{X_i \in \mathcal{X}_{pres}} rank(G_{S(\mathcal{X})}[X_i]) 
        - \sum_{e \in S(\mathcal{X}_{pres})} (d(\mathcal{X}_{pres},e) + d(\mathcal{X}_{dest},e) - 1)\\ 
        &+ \sum_{X_i \in \mathcal{X}_{dest}} rank(G_{S(\mathcal{X})}[X_i]) 
        + \sum_{X_i \in \mathcal{X}_{triv}} 1\\ 
        &- \sum_{e \in \overline{S(\mathcal{X}_{pres})}} (d(\mathcal{X}_{pres},e) + d(\mathcal{X}_{dest},e) - 1).  
    \end{align*}
    Hence, we must show that 
    \begin{align}
        \begin{split}
            \label{eqn:6}
            \sum_{e \in E(G) \setminus E(\mathcal{X}_{pres})} 1 &= 
            \sum_{X_i \in \mathcal{X}_{dest}} rank(G_{S(\mathcal{X})}[X_i]) 
            + \sum_{X_i \in \mathcal{X}_{triv}} 1\\ 
            &- \sum_{e \in S(\mathcal{X}_{pres})} (d(\mathcal{X}_{dest},e))
            - \sum_{e \in \overline{S(\mathcal{X}_{pres})}} (d(\mathcal{X}_{pres},e) + d(\mathcal{X}_{dest},e) - 1).
        \end{split}
    \end{align}
    For each $X_i \in \mathcal{X}_{triv}$, there is only one edge in $G[X_i]$, and it contributes $1$ to both the LHS and RHS of \ref{eqn:6}.  
    For each $X_i \in \mathcal{X}_{dest}$, the definition of $\mathcal{X}^s$ implies that $X_i$ does not contain the endpoints of any nonedge in $S(\mathcal{X})$, and so $G[X_i] = G_{S(\mathcal{X})}[X_i]$.  
    Consequently, the remaining contributions to the RHS are from edges $e$ of $G[X_i]$ where $X_i \in \mathcal{X}_{dest}$.  
    Note that $G[X_i]$ is independent since $G$ is, and so $e$ contributes $d(\mathcal{X}_{dest},e)$ to the first sum on the RHS.  
    If $e$ is contained in $S(\mathcal{X}_{pres})$, then it contributes $d(\mathcal{X}_{dest},e)$ to the third sum on the RHS, and $0$ overall to the RHS.  
    Otherwise, the definition of $\overline{S(\mathcal{X}_{pres})}$ shows that $d(\mathcal{X}_{pres},e)$ is either $0$ or $1$, and so $e$ contributes $1$ in the former case and $0$ in the latter case to the RHS overall.  
    This shows that the RHS is equal to the number of edges in $E(G) \setminus E(\mathcal{X}_{pres})$, which is the LHS, as desired.  

    Next, let $\mathcal{X}'_{pres} \subseteq \mathcal{X}^s$ be the set of all clusters of size at least three and let $E(\mathcal{X}'_{pres})$ be the set of edges of $G^s$ induced by the union of all sets in $\mathcal{X}'_{pres}$.  
    Since any cluster in $\mathcal{X}^s \setminus \mathcal{X}'_{pres}$ is exactly the endpoints of some edge of $G^s$ and $F$ is a set of nonedges in $G^s$, we have 
    \begin{align}
        \begin{split}
            \label{eqn:7}
            IE'(G^s,\mathcal{X}^s) &= 
            \sum_{X_i \in \mathcal{X}'_{pres}} rank(G^s_{S(\mathcal{X}^s) \cup F}[X_i]) 
            + \sum_{e \in E(G^s) \setminus E(\mathcal{X}'_{pres})} 1\\
            &- \sum_{e \in S(\mathcal{X}^s) \cup F} (d(\mathcal{X}^s,e)-1).  
        \end{split}
    \end{align}
    For any cluster $X_i \in \mathcal{X}_{pres}$ and its corresponding cluster $X'_i \in \mathcal{X}'_{pres}$, the definition of a safe \SG/ ensures that $G_{S(\mathcal{X})}[X_i]$ and $G^s_{S(\mathcal{X}^s) \cup F}[X'_i]$ are isomorphic.  
    Hence, the first two summations in \ref{eqn:5} are equal to the first two summations in \ref{eqn:7}, respectively.  
    Also, observe that 
    \begin{align*}
        &\sum_{e \in S(\mathcal{X}_{pres}) \setminus \{(a,b)\}} (d(\mathcal{X}_{pres},e) - 1) = 
        \sum_{e \in S(\mathcal{X}^s) \setminus F} (d(\mathcal{X}^s,e)-1),\\ 
        &d(\mathcal{X}_{pres},(a,b)) - 1 = \sum_{e \in F} (d(\mathcal{X}^s,e) - 1) + |F| - 1.
    \end{align*}
    Combining the above equalities completes the proof.  
\end{proof}

\begin{lemma}[IE count of safe \SG/ graph]
    \label{lem:ie-count}
    Consider a safe \SG/ yielding the graph $G^s:H$ and its cover $\mathcal{X}'$, and let $F$ be the set of key gluing pairs of $G^s$.  
    The following holds:
    $$IE(G^s:H,\mathcal{X}') = IE'(G^s,\mathcal{X}^s) + IE'(H,\mathcal{X}^H) - |F|.$$
\end{lemma}

\begin{proof}
    Let $G^{sg} = G^s:H$ and observe that 
    \begin{align}
        \begin{split}
            \label{eqn:sg-expanded}
            IE(G^{sg},\mathcal{X}') 
            &= \sum_{X_i \in \mathcal{X}^s} rank(G^{sg}_{S(\mathcal{X}')}[X_i]) + \sum_{X_i \in \mathcal{X}^H} rank(G^{sg}_{S(\mathcal{X}')}[X_i])\\ 
            &- \sum_{e \in S(\mathcal{X}')} d(\mathcal{X}^s,e) - \sum_{e \in S(\mathcal{X}')} (d(\mathcal{X}^H,e)-1).
        \end{split}
    \end{align}
    The lemma follows from the following three observations.  
    Let $F'$ be the set of key gluing pairs of $(H,\mathcal{X}^H)$.  
    First, the definition of a safe \SG/ ensures that, for any $X_i \in \mathcal{X}^s$, $G^{sg}_{S(\mathcal{X}')}[X_i] = G^s_{S(\mathcal{X}^s) \cup F}[X_i]$ and, for any $X_i \in \mathcal{X}^H$, $G^{sg}_{S(\mathcal{X}')}[X_i] = H_{S(\mathcal{X}^H) \cup F'}[X_i]$.  
    Second, since each pair in $S(\mathcal{X}') \setminus F$ corresponds to some pair in either $S(\mathcal{X}^s) \setminus F$ or $S(\mathcal{X}^H) \setminus F'$ but not both, we have 
    \begin{align*}
        &\sum_{e \in S(\mathcal{X}') \setminus F} d(\mathcal{X}^s,e)
        + \sum_{e \in S(\mathcal{X}') \setminus F} (d(\mathcal{X}^H,e)-1)\\ 
        &= \sum_{e \in S(\mathcal{X}^s) \setminus F} (d(\mathcal{X}^s,e)-1) + \sum_{e \in S(\mathcal{X}^H) \setminus F'} (d(\mathcal{X}^H,e)-1).  
    \end{align*}
    Third, for any pair $e \in F$ and the corresponding pair $e' \in F'$, we have $d(\mathcal{X}',e) - 1 = (d(\mathcal{X}^s,e) - 1) + (d(\mathcal{X}^H,e') - 1) + 1$.  
    Note that the ``$+1$'' term in this last observation is responsible for the ``$-|F|$'' term in the Lemma statement.  
    This completes the proof.  
\end{proof}

Finally, we prove Theorem \ref{thm:sg-implied}.  

\begin{proof}[Proof of Theorem \ref{thm:sg-implied}]
    Consider any safe \SG/ yielding the graph $G^s:H$ and its cover $\mathcal{X}'$, and let $G^{sg} = G^s:H$.  
    The definitions of a safe ear and safe \SG/ ensure that some key gluing pair of $G^s$ is a nonedge of $G^{sg}$ that is contained in the shared set $S(\mathcal{X}')$.  
    We will show that any nonedge $f$ in $S(\mathcal{X}')$ is implied in $G^{sg}$ using the rank-sandwich technique.  
    Since Lemma \ref{lem:sg-cover-ind} shows that $\mathcal{X}'$ is independent, and $\mathcal{X}'$ is clearly a $2$-thin cover of $G^{sg} \cup f$, Lemma \ref{thm:2-thin-rank} shows that both $rank(G^{sg})$ and $rank(G^{sg} \cup f)$ are upper-bounded by $IE(G^{sg},\mathcal{X}')$.  
    Additionally, Lemmas \ref{lem:edge-count} and \ref{lem:ie-count} show that $rank(G^{sg}) = IE(G^{sg},\mathcal{X}')$.  
    Thus, we have $rank(G^{sg}) = rank(G^{sg} \cup f)$, and so $f$ is implied in $G^{sg}$.  
\end{proof}

\subsection{Constructing safe base graphs}
\label{sec:safe_base}

Before we prove Theorem \ref{thm:starter}, we give several examples of \emph{seed graphs}, which are small fixed size safe base graphs that cannot be constructed using any of the inductive constructions in the previous sections.  
The motivation for these examples is to provide explicit graphs that can be enlarged by repeatedly applying Theorem \ref{thm:starter}.  
%
\begin{itemize}
    \item \textbf{Modified octahedral ring:} a ring of between 7 and 10 octahedral graphs, as shown on the left in Figure \ref{fig:octaForstarting}, such that the double-dashed edges are deleted and used as hinges.  
    
    \item \textbf{Modified icosahedral ring:} a ring of between $7$ and $10$ icosahedral graphs, as shown on the right in Figure \ref{fig:octaForstarting}, such that one double-dashed edge is deleted, and the hinges are the deleted edge and a dashed nonedge that does not share a vertex with the edge.  

    \item \textbf{Two icosahedral-sharing rings:} two modified icosahedral 7-rings that intersect on two consecutive links, as in Figure \ref{fig:ring_k12} where links are circles and hinges are nonedges shown as dashed line-segments.  
    For the lower shared icosahedral link, the edges deleted and used as hinges in the left and right rings are distinct and vertex-disjoint.  

    \item \textbf{Four icosahedral-sharing rings:} refer to the graph in Figure \ref{fig:otherScheme}, which is constructed from four rings with 8 links each.  
    The small circles are butterflies whose hinges are nonedges that do not share vertices.  
    The four large circles are icosahedral graphs.  
    The hinges of $G_1$ and $G_3$ are the three double-dashed edges and one dashed nonedge that does not share a vertex with these edges in Figure \ref{fig:octaForstarting}.  
    The hinges of $G_2$ and $G_4$ are one double-dashed edge and three dashed nonedges
    No two hinges share a vertex and the double-dashed edges used as hinges are deleted.  
    Note that $G_1$ and $G_3$ each have $27$ edges and exactly $3$ independent flexes, and $G_2$ and $G_4$ each have $29$ edges and exactly $1$ independent flex.  
\end{itemize}
For each of the above seed graphs, the properties of a safe base graph are verified as follows.  
Since these are relatively small graphs, independence can be checked either by computing the rank of a random framework or via symbolic computation, to check that at least one monomial of the symbolic determinantal polynomial has a nonzero coefficient.  
However, the former method of checking if a single scalar determinant value is nonzero is subject to rounding errors in the numerical computation and while the latter could be intractable even for graphs of these small sizes if coordinates are assigned distinct variables,   
 it can be made tractable for an ansatz framework in a special position with several vertex coordinate variables equal, which results in  tractable monomial expansion of the symbolic determinantal polynomials. 
Next, for each seed graph, taking the collection of vertex sets of all links is clearly an independent $2$-thin cover.  
Lastly, the rank condition follows from independence, which implies the rank is equal to the number of edges, and Lemma \ref{thm:2-thin-rank}.  

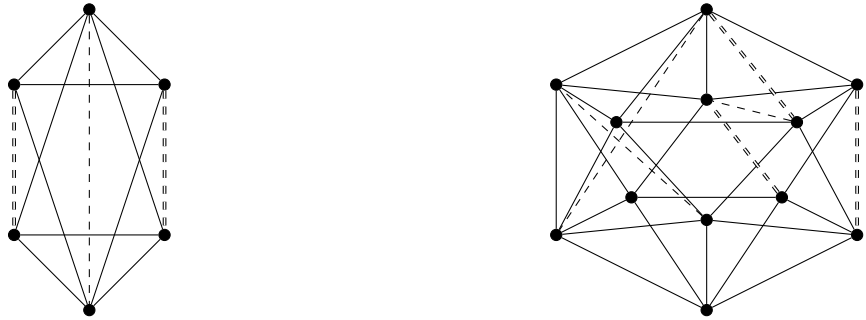
\begin{figure}[htb]
    \centering
    \begin{subfigure}[t]{0.49\linewidth}
        \centering
        \begin{tikzpicture}
            \node[draw,circle,fill,inner sep=-1.5pt] (a1) at (0,0) {};

            \node[draw,circle,fill,inner sep=-1.5pt,] (c1) at (1,1) {};

            \node[draw,circle,fill,inner sep=-1.5pt] (a2) at (2,0) {};

            \node[draw,circle,fill,inner sep=-1.5pt] (b1) at (0,-2) {};

            \node[draw,circle,fill,inner sep=-1.5pt] (b2) at (2,-2) {};

            \node[draw,circle,fill,inner sep=-1.5pt] (c2) at (1,-3) {};

            \draw (c1) -- (a1);
            \draw (c1) -- (a2);
            \draw (c1) -- (b1);
            \draw (c1) -- (b2);
            \draw[dashed] (c1) -- (c2);

            \draw (c2) -- (a1);
            \draw (c2) -- (a2);
            \draw (c2) -- (b1);
            \draw (c2) -- (b2);

            \draw (a1) -- (a2);
            \draw (b1) -- (b2);

            \draw[double, dashed] (a1) -- (b1);
            \draw[double, dashed] (a2) -- (b2);
        \end{tikzpicture}
    \end{subfigure}
    \begin{subfigure}[t]{0.49\linewidth}
        \centering
        \begin{tikzpicture}
            \node[draw,circle,fill,inner sep=-1.5pt] (v1) at (0,0) {};
            \node[draw,circle,fill,inner sep=-1.5pt] (v2) at (0,-2) {};
            \node[draw,circle,fill,inner sep=-1.5pt] (v3) at (-2,1) {};
            \node[draw,circle,fill,inner sep=-1.5pt] (v4) at (-0.8,-0.5) {};
            \node[draw,circle,fill,inner sep=-1.5pt] (v5) at (-1,-1.5) {};
            \node[draw,circle,fill,inner sep=-1.5pt] (v6) at (-2,-0.2) {};
            \node[draw,circle,fill,inner sep=-1.5pt] (v7) at (-3,-1.5) {};
            \node[draw,circle,fill,inner sep=-1.5pt] (v8) at (-3.2,-0.5) {};
            \node[draw,circle,fill,inner sep=-1.5pt] (v9) at (-2,-1.8) {};
            \node[draw,circle,fill,inner sep=-1.5pt] (v10) at (-2,-3) {};
            \node[draw,circle,fill,inner sep=-1.5pt] (v11) at (-4,-2) {};
            \node[draw,circle,fill,inner sep=-1.5pt] (v12) at (-4,0) {};

            \draw[double, dashed] (v1) -- (v2);
            \draw (v1) -- (v3);
            \draw (v1) -- (v4);
            \draw (v1) -- (v5);
            \draw (v1) -- (v6);

            \draw (v2) -- (v4);
            \draw (v2) -- (v5);
            \draw (v2) -- (v9);
            \draw (v2) -- (v10);

            \draw[double, dashed] (v3) -- (v4);
            \draw (v3) -- (v6);
            \draw (v3) -- (v8);
            \draw (v3) -- (v12);

            \draw (v4) -- (v8);
            \draw (v4) -- (v9);

            \draw[double, dashed] (v5) -- (v6);
            \draw (v5) -- (v7);
            \draw (v5) -- (v10);

            \draw (v6) -- (v7);
            \draw (v6) -- (v12);

            \draw (v7) -- (v10);
            \draw (v7) -- (v11);
            \draw (v7) -- (v12);

            \draw (v8) -- (v9);
            \draw (v8) -- (v11);
            \draw (v8) -- (v12);

            \draw (v9) -- (v10);
            \draw (v9) -- (v11);

            \draw (v10) -- (v11);

            \draw (v11) -- (v12);

            \draw[dashed] (v9) -- (v12);
            \draw[dashed] (v3) -- (v11);
            \draw[dashed] (v4) -- (v6);
        \end{tikzpicture}
    \end{subfigure}
    \caption{Left: an octahedral graph with two double-dashed edges and one dashed nonedge.  
    Right: an icosahedral graph with three double-dashed edges and one dashed nonedge.  
    See the example seed graphs in Section \ref{sec:safe_base}.}
    \label{fig:octaForstarting}
\end{figure}

\begin{figure}[htb]
    \centering
    \begin{tikzpicture}[scale=0.5]
        \draw (0,0) circle (1.5);
        \draw (0,-2) circle (1.5);
        
        \draw (-1.6,-3.8) circle (1.5);
        \draw (-3.7,-3.3) circle (1.5);
        \draw (-4.2,-1) circle (1.5);
        \draw (-3.2,1.3) circle (1.5);
        \draw (-1,1.7) circle (1.5);

        \draw (1.6,-3.8) circle (1.5);
        \draw (3.7,-3.3) circle (1.5);
        \draw (4.2,-1) circle (1.5);
        \draw (3.2,1.3) circle (1.5);
        \draw (1,1.7) circle (1.5);

        \draw[dashed] (-0.5,-1) -- (0.5,-1);
        \draw[dashed] (0,0.9) -- (0,1.4);
        
        \draw[dashed] (-1.2,-2.6) -- (-0.5,-3.2);
        \draw[dashed] (-2.5,-2.9) -- (-2.8,-4.1);
        \draw[dashed] (-4.5,-2.3) -- (-3.3,-2);
        \draw[dashed] (-4.2,0.38) -- (-3.3,0);
        \draw[dashed] (-2.2,2.1) -- (-2,1);

        \draw[dashed] (1.2,-2.6) -- (0.5,-3.2);
        \draw[dashed] (2.5,-2.9) -- (2.8,-4.1);
        \draw[dashed] (4.5,-2.3) -- (3.3,-2);
        \draw[dashed] (4.2,0.38) -- (3.3,0);
        \draw[dashed] (2.2,2.1) -- (2,1);
    \end{tikzpicture}
    \caption{Two icosahedral-sharing rings which form a seed graph, as discussed in Section \ref{sec:safe_base}.}
    \label{fig:ring_k12}
\end{figure}
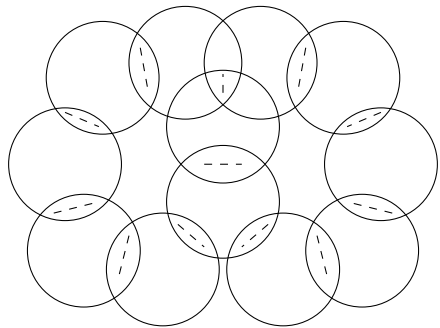

\begin{figure}[htb]
    \centering
    \begin{tikzpicture}[scale=0.5]
        \draw (0,4) circle (2);
        \node at (0,4) {\huge$G_1$};
        
        \draw (0,-4) circle (2);
        \node at (0,-4) {\huge$G_3$};
        
        \draw (-4,0) circle (2);
        \node at (-4,0) {\huge$G_2$};
        
        \draw (4,0) circle (2);
        \node at (4,0) {\huge$G_4$};

        \draw (-2.2,3.5) circle (0.7);
        \draw (-3,2.8) circle (0.7);
        \draw (-3.8,2.1) circle (0.7);

        \draw (-1,2) circle (0.7);
        \draw (-1,1) circle (0.7);
        \draw (-2,0.8) circle (0.7);

        \draw (-2.2,-3.5) circle (0.7);
        \draw (-3,-2.8) circle (0.7);
        \draw (-3.8,-2.1) circle (0.7);

        \draw (-1,-2) circle (0.7);
        \draw (-1,-1) circle (0.7);
        \draw (-2,-0.8) circle (0.7);

        \draw (2.2,-3.5) circle (0.7);
        \draw (3,-2.8) circle (0.7);
        \draw (3.8,-2.1) circle (0.7);

        \draw (1,-2) circle (0.7);
        \draw (1,-1) circle (0.7);
        \draw (2,-0.8) circle (0.7);

        \draw (2.2,3.5) circle (0.7);
        \draw (3,2.8) circle (0.7);
        \draw (3.8,2.1) circle (0.7);

        \draw (1,2) circle (0.7);
        \draw (1,1) circle (0.7);
        \draw (2,0.8) circle (0.7);

        \draw[dashed] (-1.9,4) -- (-1.7,3.2);
        \draw[dashed] (-2.8,3.4) -- (-2.4,2.9);
        \draw[dashed] (-3.6,2.7) -- (-3.2,2.2);
        \draw[dashed] (-4.3,1.85) -- (-3.4,1.7);

        \draw[dashed] (-1.2,2.55) -- (-0.45,2.2);
        \draw[dashed] (-1.3,1.5) -- (-0.7,1.5);
        \draw[dashed] (-1.55,1.2) -- (-1.45,0.7);
        \draw[dashed] (-2.5,1.1) -- (-2.2,0.4);

        \draw[dashed] (-1.9,-4) -- (-1.7,-3.2);
        \draw[dashed] (-2.8,-3.4) -- (-2.4,-2.9);
        \draw[dashed] (-3.6,-2.7) -- (-3.2,-2.2);
        \draw[dashed] (-4.3,-1.85) -- (-3.4,-1.7);

        \draw[dashed] (-1.2,-2.55) -- (-0.45,-2.2);
        \draw[dashed] (-1.3,-1.5) -- (-0.7,-1.5);
        \draw[dashed] (-1.55,-1.2) -- (-1.45,-0.7);
        \draw[dashed] (-2.5,-1.1) -- (-2.2,-0.4);

        \draw[dashed] (1.9,-4) -- (1.7,-3.2);
        \draw[dashed] (2.8,-3.4) -- (2.4,-2.9);
        \draw[dashed] (3.6,-2.7) -- (3.2,-2.2);
        \draw[dashed] (4.3,-1.85) -- (3.4,-1.7);

        \draw[dashed] (1.2,-2.55) -- (0.45,-2.2);
        \draw[dashed] (1.3,-1.5) -- (0.7,-1.5);
        \draw[dashed] (1.55,-1.2) -- (1.45,-0.7);
        \draw[dashed] (2.5,-1.1) -- (2.2,-0.4);

        \draw[dashed] (1.9,4) -- (1.7,3.2);
        \draw[dashed] (2.8,3.4) -- (2.4,2.9);
        \draw[dashed] (3.6,2.7) -- (3.2,2.2);
        \draw[dashed] (4.3,1.85) -- (3.4,1.7);

        \draw[dashed] (1.2,2.55) -- (0.45,2.2);
        \draw[dashed] (1.3,1.5) -- (0.7,1.5);
        \draw[dashed] (1.55,1.2) -- (1.45,0.7);
        \draw[dashed] (2.5,1.1) -- (2.2,0.4);

    \end{tikzpicture}
    \caption{Four icosahedral-sharing rings which form a seed graph, as discussed in Section \ref{sec:safe_base}.}
    \label{fig:otherScheme}
\end{figure}
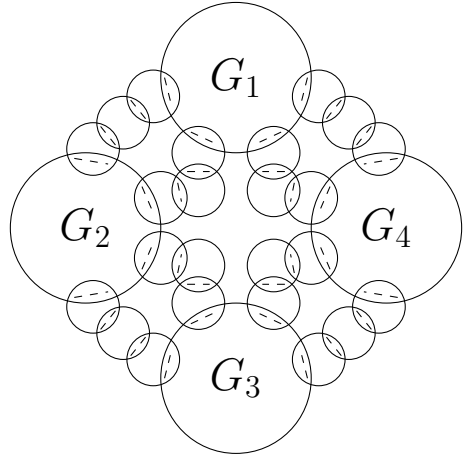

\begin{proof}[Proof of Theorem \ref{thm:starter}]
    Consider a safe double-butterfly \SG/, yielding the graph $G^s:H$ and its cover $\mathcal{X}'$, that splits a nonedge $(a,b)$ into the nonedges $(a_1,b_1)$ and $(a_2,b_2)$.
    Since $G$ is independent, Theorem \ref{thm:drsg-ind} shows that $G^s:H$ is independent.  
    Also, Lemma \ref{lem:sg-cover-ind} shows that $\mathcal{X}'$ is independent.  
    Note that we can choose the safe ear cover of a double-butterfly to contain exactly two clusters, one for the vertex set of each butterfly.  
    Hence, the shared set $S(\mathcal{X'})$ contains the nonedge $(u,v)$ in Figure \ref{fig:twoRoofs}, and a safe nonedge-split can clearly be performed on $(G^s:H,\mathcal{X}')$ and $(u,v)$.  
    Lastly, Lemmas \ref{lem:edge-count} and \ref{lem:ie-count} show that $rank(G^s:H) = IE(G^s:H,\mathcal{X}')$.  
    The above facts show that $G^s:H$ is a safe base graph.  

    Next, let $G'$ be obtained via a $k$-sum on a safe base graph $G$ and an independent graph $H$.  
    Since $G$ is independent and $k$-sums preserve independence, $G'$ is independent.  
    Let $\mathcal{X}$ be a $2$-thin cover such that a safe nonedge-split can be performed on $(G,\mathcal{X})$ and let $\mathcal{X}' = \mathcal{X} \cup I$, where $I$ is the collection of all sets $\{u,v\}$ such that $(u,v)$ is an edge of $H$ not in its base complete graph.  
    It is easy to see that $\mathcal{X}'$ is a $2$-thin cover of $G'$ such that $S(\mathcal{X}') = S(\mathcal{X})$, and so $\mathcal{X}'$ is independent since $\mathcal{X}$ is independent.  
    Hence, since a safe nonedge-split can be performed on $(G,\mathcal{X})$ and some nonedge $(a,b)$ in $S(\mathcal{X})$, one can also be performed on $(G',\mathcal{X}')$ and $(a,b)$.  
    Lastly, observe that $rank(G') = rank(G) + |I|$, and the definition of $\mathcal{X'}$ makes it clear that $IE(G',\mathcal{X}') = IE(G,\mathcal{X}) + |I|$.  
    Therefore, since $rank(G) = IE(G,\mathcal{X})$, we get that $rank(G') = IE(G',\mathcal{X}')$, and so $G'$ is a safe base graph.  
    
    Finally, let $G'$ be obtained via a Henneberg-I extension on a safe base graph $G$.  
    The proof is almost identical to the previous case, except for how $\mathcal{X}'$ is constructed from $\mathcal{X}$, so we omit all but this part of the proof.  
    $\mathcal{X}'$ is constructed from $\mathcal{X}$ by adding $u$ to some cluster that contains all neighbors of $u$ if such a cluster exists, or adding to $\mathcal{X}$ each set $\{u,x\}$ such that $u$ neighbors $x$ otherwise.  
    This completes the proof.  
\end{proof}

\section{Constructing dependent nucleation-free graphs with implied nonedges}
\label{sec:dependent}

Here we prove Theorem \ref{thm:dependent} and then give two examples of its application.  

\begin{proof}[Proof of Theorem \ref{thm:dependent}]
    Let $G = G_1 \cup G_2$.  
    For (i), the fact that $G_1$ and $G_2$ share only the endpoints of $f$ implies that any subgraph of $G$ either is a subgraph of $G_1$ or $G_2$ but not both or is separated by the endpoints of $f$.  
    Hence, since $G_1$ and $G_2$ are nucleation-free, we get that $G$ is nucleation-free.  
    For (ii), since $f$ is implied in both $G_1$ and $G_2$, the circuit elimination axiom shows that $G$ is dependent.  
    Lastly, for (iii), $G$ is dependent by (ii), and we show $G \setminus e$ is independent for any edge $e$ of $G$ as follows.  
    If $G_1 \cup f$ and $G_2 \cup f$ are circuits, then $(G_1 \cup f) \setminus e_1$ and $(G_2 \cup f) \setminus e_2$ are independent graphs.  
    Since, by Lemma \ref{lem:abs}, $2$-sums preserve independence in any abstract $3$-rigidity matroid, the graph $G'$ resulting from the $2$-sum of these two graphs is independent.  
    Wlog assume $e$ is $e_1$.  
    Observe that $G' \cup e_2$ contains a unique circuit, which involves $f$, and so $(G' \cup e_2) \setminus f$ must be independent.  
    Since $G \setminus e = (G' \cup e_2) \setminus f$, we get that $G$ is a circuit.  
    Note that the above proof shows (iii) holds for any abstract $d$-rigidity matroid.  
\end{proof} 

Note that the while theorem is very general, there are infinitely many concrete example graphs to which it can be applied, and as the theorems in Section \ref{sec:contributions} show, all of these example graphs can be extended to infinite families using any of the inductive constructions discussed in this paper.  
Below are two well-known and easily drawn graphs to which Theorem \ref{thm:dependent} applies.  
\begin{figure}[htb]
    \centering
    \begin{subfigure}[t]{0.49\linewidth}
        \centering
        \includegraphics[width=0.7\textwidth]{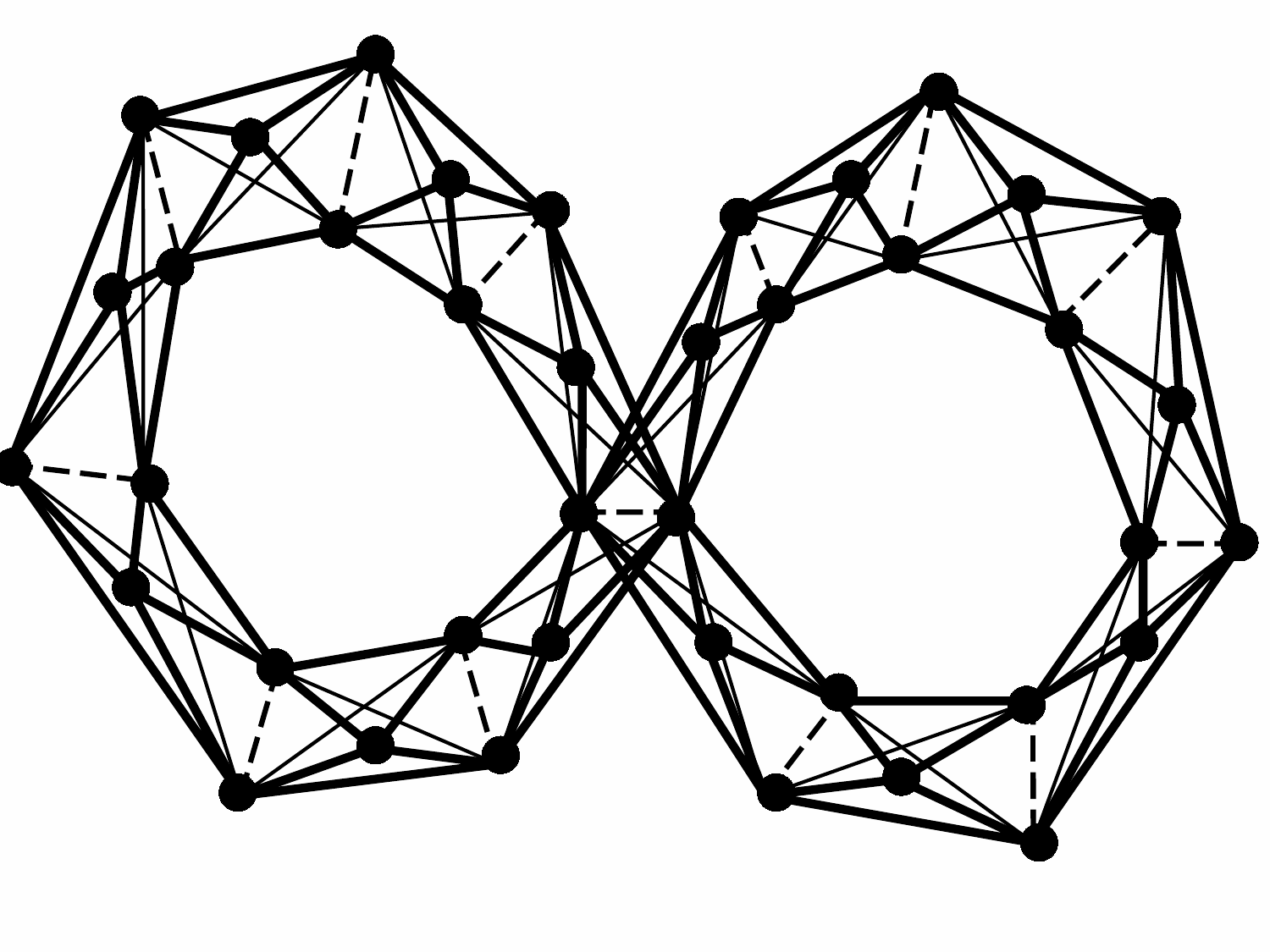}
    \end{subfigure}
    \begin{subfigure}[t]{0.49\linewidth}
        \centering
        \includegraphics[width=0.7\textwidth]{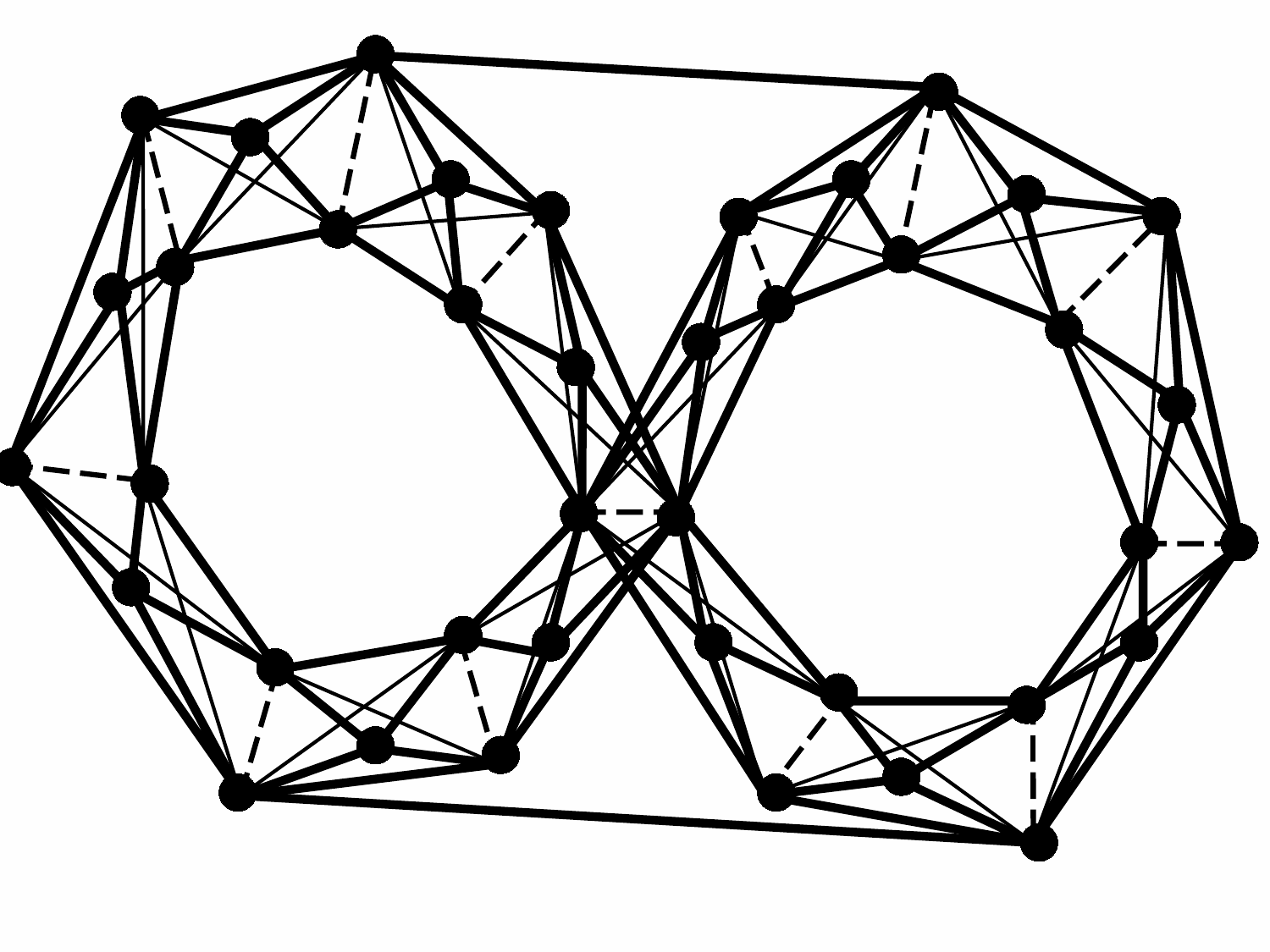}
    \end{subfigure}
    \caption{Left: a graph $G$ consisting of two rings of butterflies sharing a hinge and each with $7$ links.  
    Right: adding two edges to $G$ as shown yields a $(3,6)$-tight graph.  
    See the example applications of Theorem \ref{thm:dependent} in Section \ref{sec:dependent}.}
    \label{fig:doubleRing}
\end{figure}
\begin{itemize}
    \item For the graph $G$ consisting of two rings of butterflies sharing a hinge on the left in Figure \ref{fig:doubleRing}, applying Theorems \ref{thm:ring_of_roofs} and \ref{thm:dependent} in sequence shows that $G$ is dependent and nucleation-free.  

    \item Adding the edges to $G$ shown on the right in Figure \ref{fig:doubleRing} results in a $(3,6)$-tight graph, which is dependent and flexible since it has $G$ as a subgraph.  
\end{itemize}

\section{Open problems}
\label{sec:conclusion}

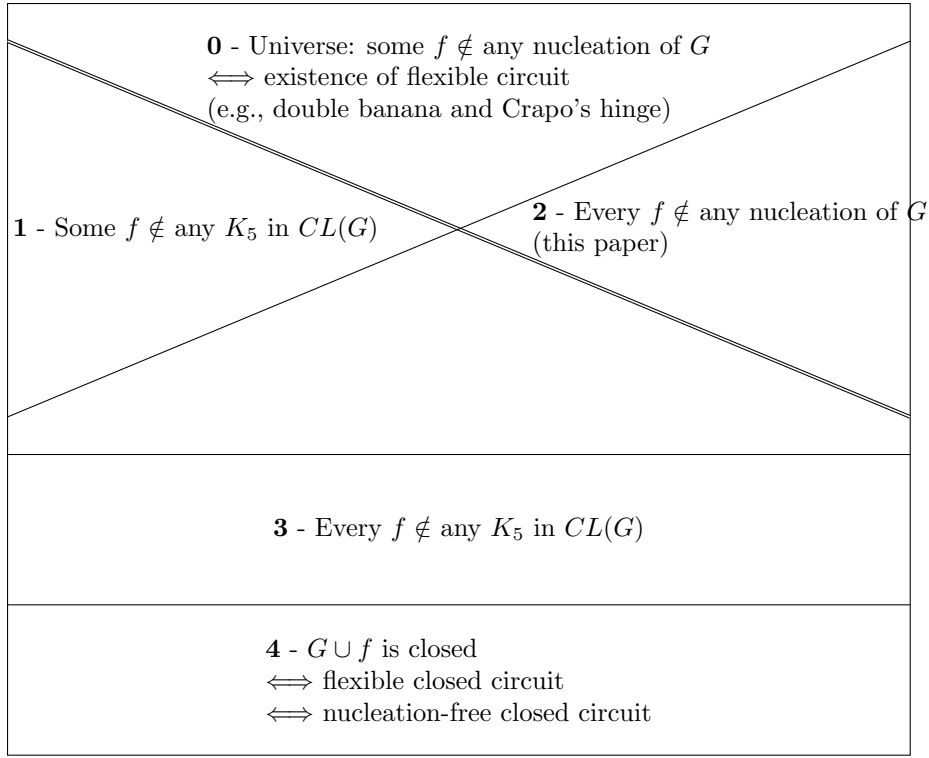
\begin{figure}
    \centering
    \begin{tikzpicture}
        \node[align=left] at (6,10.5) {Assumptions: $G$ independent graph, $|V(G)| \geq 6$, $f$ implied nonedge of $G$.};
        
        \draw (0,0) rectangle (12,10);
        
        \node[align=left] at (6,9) {\textbf{0} - Universe: some $f \notin$ any nucleation of $G$\\ 
        $\Longleftrightarrow$ existence of flexible circuit\\ 
        (e.g., double banana and Crapo's hinge)};

        \node[align=left] at (9.6,7) {\textbf{2} - Every $f \notin$ any nucleation of $G$\\ 
        (this paper)};

        \draw[double] (0,9.5) -- (12,4.5);

        \node[align=left] at (2.5,7) {\textbf{1} - Some $f \notin$ any $K_5$ in $CL(G)$};

        \draw (0,4.5) -- (12,9.5);

        \draw (0,4) -- (12,4);

        \node[align=left] at (6,3) {\textbf{3} - Every $f \notin$ any $K_5$ in $CL(G)$};

        \draw (0,2) -- (12,2);

        \node[align=left] at (6,1) {\textbf{4} - $G \cup f$ is closed\\ 
        $\Longleftrightarrow$ flexible closed circuit\\ 
        $\Longleftrightarrow$ nucleation-free closed circuit};
    \end{tikzpicture}
    \caption{Each block in the diagram represents a set of graphs defined by the labels.  
    Higher indexed sets are subsets of lower indexed sets, except possibly sets 1 and 2 whose subset relation is unknown.  
    It is unknown whether the sets below the double line  are non-empty, and non-emptiness of any of them would disprove the maximality conjecture.  
    See the open problems in Section \ref{sec:conclusion}.  }
    \label{fig:open}
\end{figure}

\begin{itemize}
\item
One set of key open problems is shown in Figure \ref{fig:open}. Specifically, showing emptiness  of any of the classes below the double line would represent progress toward proving Graver's maximality conjecture, whereas showing nonemptiness would refute it.
\item
Given an oracle that recognizes  any  arbitrarily large seed graphs  such as those indicated  by Section \ref{sec:safe_base} 
or the links in Theorem \ref{thm:henneberg-ii_ring} and the remark thereafter, 
give a polynomial time algorithm to  recognize that an input graph is the result of (nucleation-free, independent graph, with implied nonedge) constructions in this paper applied to such a seed graph. 
 \item
Extend the application of the flex-sign technique, i.e, find other graphs besides the butterfly that satisfy the expansion/contraction property.  
 \item
Nail down the remainder of the proof of that known inductive constructions cannot be combined to replicate double-butterfly \SG/.
\item
To complete our understanding of nucleation-free graphs with implied nonedges,  study examples extending those in  Section  \ref{sec:safe_base} that cannot be obtained by any of our construction schemes. 
\item
Another interesting open problem is to extend our double-butterfly \SG/ so that it additionally preserves the dimension of the flex space.  
In order to do that, we need to add two edges to the double-butterfly.  
One possible way is to add two more edges $(c', a_1)$, $(c, b_2)$ (or $(c', b_1)$, $(c, a_2))$.  
We note that our current method to show independence in Theorem \ref{thm:drsg-ind} would fail as  there is a non-zero stress on the added part.  
However, if we can show that  the existence of a generic circuit in the new graph $G$ implies there  are edges   $\{(w_1, t_1), \ldots (w_n, t_n)\}$ that remain dependent for the non-generic position $\p$ used in Theorem \ref{thm:drsg-ind}, i.e., there exists non-zero stresses $\{s_1, \ldots, s_n\}$ s.t. $\sum_i s_i (\p(w_i)-\p(t_i))$ $=$ $0$, then we can simply use Lemma \ref{lem:dra-c} to complete our proof.
\end{itemize}

\bibliographystyle{plain}

\appendix

\section{Issues in Tay's paper \cite{taybar:1993}}\label{sec:Taycounter}

This paper proves the existence of large families of graphs that are independent, have implied nonedges, and are nucleation-free, and provides inductive constructions for starting from such graphs and generating larger ones.  
As mentioned in Section \ref{sec:introduction}, independent graphs with implied nonedges have been conjectured and written down by many (\cite{taybar:1993}, \cite{jackson:jordan:rank3dRigidity:egres-05-09:2005}).  
However, to the best of our knowledge, we are the first to give proofs, and the first to study the nucleation-free property.  

In particular, Tay presented general classes of rings and claimed that they are circuits after adding an edge between the endpoints of a hinge nonedge \cite[Theorem 4.7]{taybar:1993}, i.e., the ring without this hinge edge is independent and the hinge nonedge is implied.  
Tay noted that sufficiently large rings of butterflies are contained in one of these classes.  
However, Tay's proof of this theorem  is imprecise at best.  
In \cite[Proposition 4.6]{taybar:1993}, a key tool in the proof of the theorem, Tay considered a \emph{chain} 
$C_m = C(G_1,\ldots,G_m)$ where $G_i$ and $G_{i+1}$ are graphs on at least three vertices each whose intersection is a nonedge and where $G_1$ and $G_m$ contain edges $(p,q)$ and $(r,s)$, respectively; see Figure \ref{fig:counterTay1}.  
Tay claimed that if $C_m$ is a circuit, then the ring $R_m$ obtained from $C_m$ by identifying $p$ with $r$ and $q$ with $s$ is dependent and $(p,q)$ is implied in $R_m \setminus (p,q)$.  
Note that Crapo's Hinge is an example of a chain that is a circuit, and the corresponding ring is a ring of butterflies.  
His proof of the claim starts with a generic framework of $R_m$ and a framework of $C_m$ that agrees with the framework of $R_m$, i.e., in which $p$ coincides with $r$ and $q$ coincides with $s$.  
Since $C_m$ is a circuit, its framework has a non-zero equilibrium stress $\lambda$.  
If the sum of the stresses $\lambda_{pq}$ and $\lambda_{rs}$ on the edges $(p,q)$ and $(r,s)$, respectively, is not equal to $0$, then the generic framework of $R_m$ has a non-zero equilibrium stress supported on $(p,q)$, which proves the claim.  

\begin{figure}[htb]
    \centering
    \includegraphics[width=0.6\textwidth]{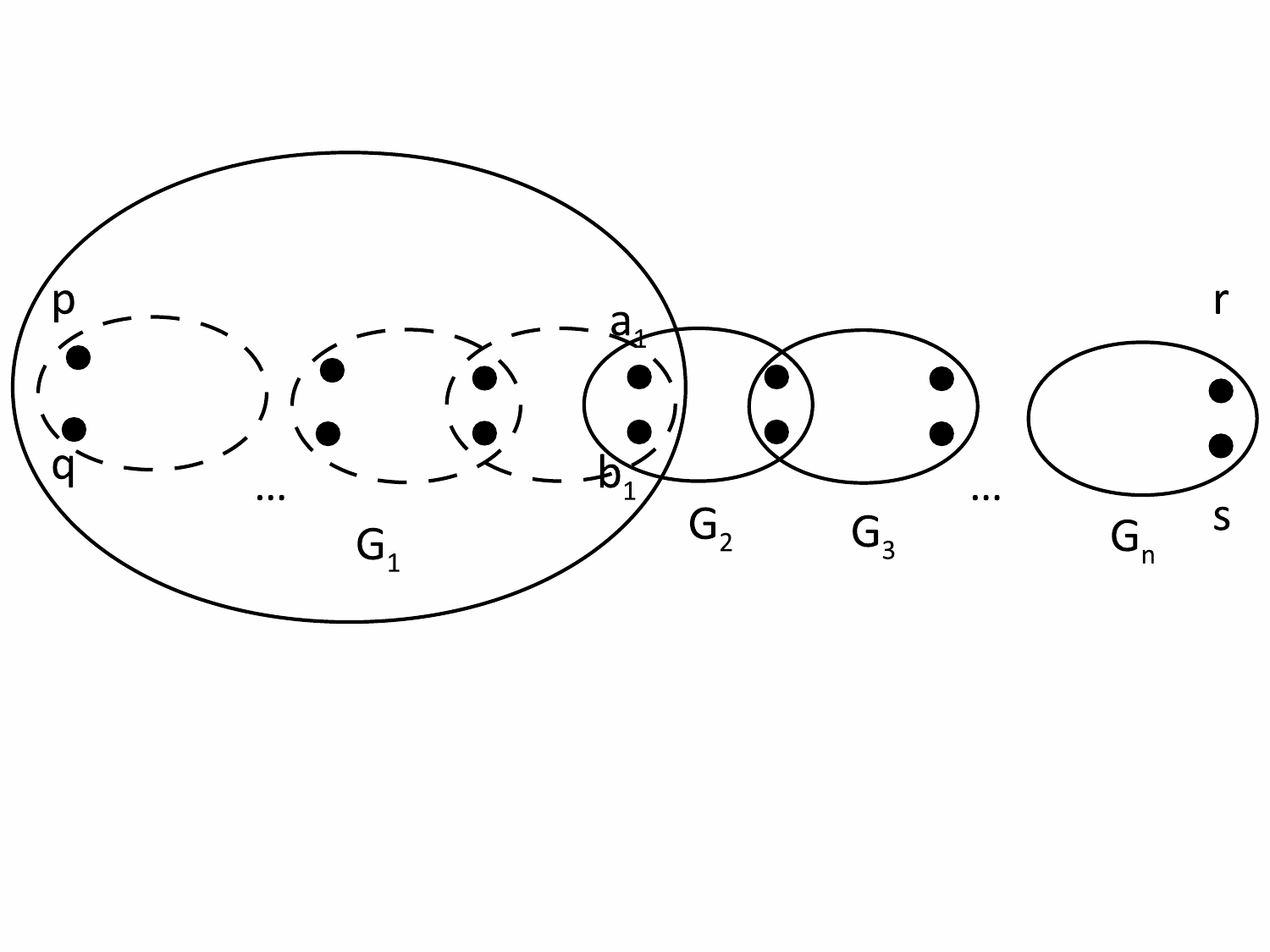}
    \caption{A chain illustrating a hole in Tay's proof discussed in Appendix \ref{sec:Taycounter}.  
    }
    \label{fig:counterTay1}
\end{figure}

\begin{figure}[!htbp]
\centering
\scalebox{0.4}[0.35]{
\includegraphics{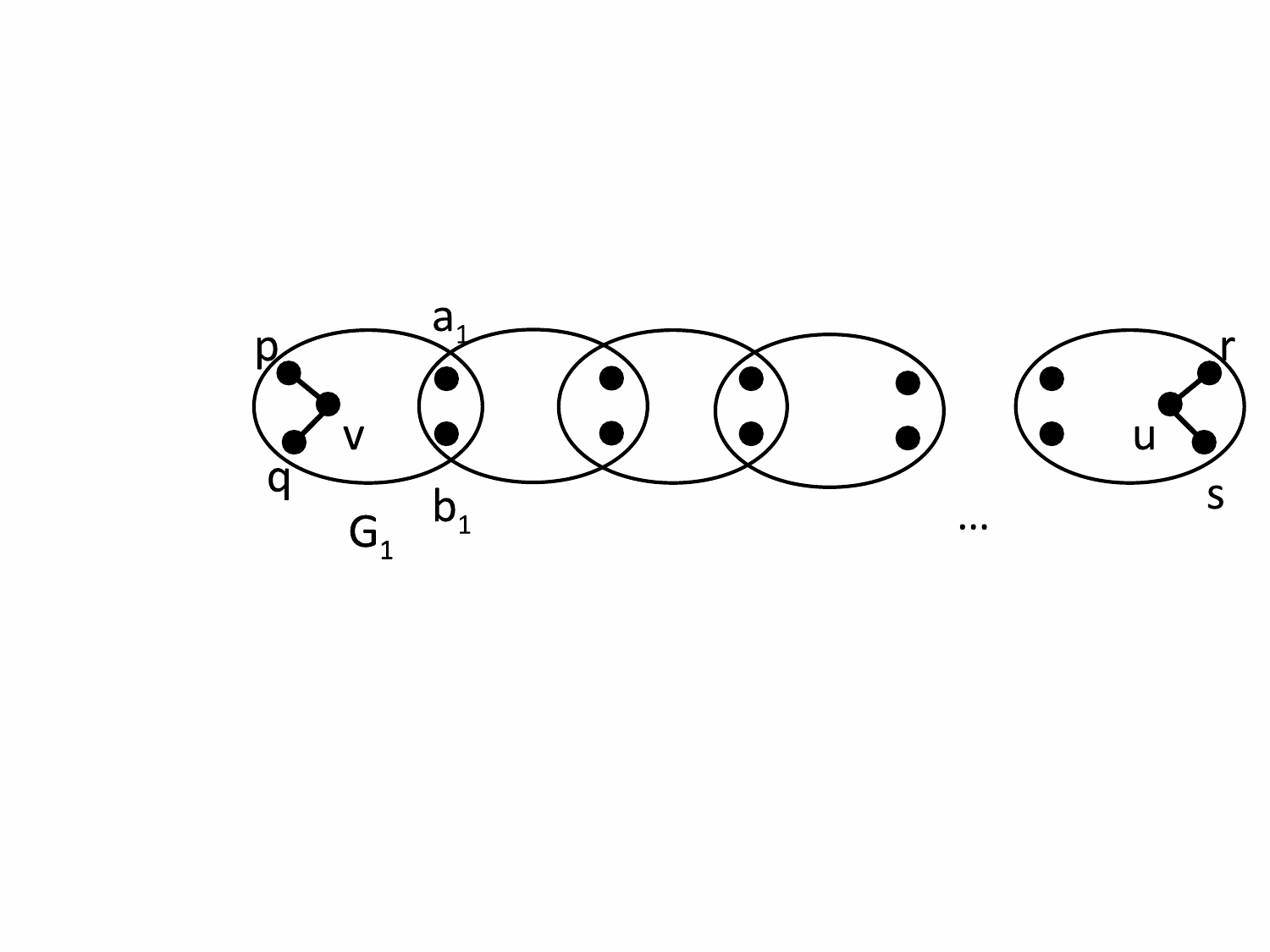}}
\caption{A framework of the chain in Figure \ref{fig:counterTay1} where all vertices adjacent to $p$ or $q$ are at the same position as $v$, all vertices adjacent to $r$ or $s$ are at the same position as $u$, the pairs $(p,q)$, $(p,v)$, and $(q,v)$ form an isosceles right triangle, and the pairs $(r,s)$, $(r,u)$, and $(s,u)$ form an isosceles right triangle.  
See Appendix \ref{sec:Taycounter}.  }\label{fig:counterTay2}
\end{figure}

On the other hand, if the sum of the stresses $\lambda_{pq}$ and $\lambda_{rs}$ on the edges $(p,q)$ and $(r,s)$, respectively, is equal to $0$, then Tay stated that ``one can keep $\lambda_{pq}$ fixed and change the value of $\lambda_{rs}$ by changing the position of either $a_1$ or $b_1$'' ($a_1$ and $b_1$ are the two vertices shared by the first two graphs $G_1$ and $G_2$).
However, he did not mention how to change the position of $a_1$ or $b_1$ to achieve this, and it is not clear that this is always possible.  
In Fig.~\ref{fig:counterTay1}, $G_1$ is itself a chain, and the illustrated framework of $C_m$ is symmetric about the line containing $a_1$ and $b_1$.  
If we move $a_1$ or $b_1$ along this line, then $\lambda_{pq}$ and $\lambda_{rs}$ always cancel out.  
The framework of $C_m$ in Fig.~\ref{fig:counterTay2} is such that $a_1$ and $b_1$ are not adjacent to $p$ or $q$, all vertices adjacent to $p$ or $q$ are at the same position (indicated as $v$), and $p$, $q$, and $v$ are in such a position that the angle between $(p,v)$ and $(p,q)$ is $45$ degrees and the angle between $(q,v)$ and $(p,q)$ is $45$ degrees.  
 In this case, it is easy to see that $\lambda_{pq} = 0$.  
 We can similarly ensure that $\lambda_{rs} = 0$.  
 Now, no matter how we move $a_1$ or $b_1$, $\lambda_{pq}$ and $\lambda_{rs}$ always cancel out.  
 While these frameworks are not in the special position discussed in Tay's argument, they illustrate holes in the proof.





\end{document}